\documentclass[12pt]{amsart}
\usepackage[utf8]{inputenc}

\usepackage{amsmath}
\usepackage{amssymb}
\usepackage[arrow,matrix,curve,cmtip,ps]{xy}
\usepackage{paralist}
\usepackage[left=1in,top=1in,right=1in,bottom=1in]{geometry}
\usepackage{mathrsfs}
\usepackage{url}
\usepackage[colorlinks=true, linktocpage=true]{hyperref}
\hypersetup{linkcolor=black}
\usepackage{scalefnt}
\usepackage{float}
\usepackage{subcaption}
\usepackage{tikz}
\usepackage{tikz-cd}
\usepackage[normalem]{ulem}
\usetikzlibrary{arrows,chains,matrix,positioning,scopes,calc,decorations,
decorations.pathreplacing}
\usepackage{comment}
\usepackage{aliascnt}
\usepackage{bbm}

\usepackage{xcolor}

\newcommand{\R}{\mathbb{R}}
\newcommand{\N}{\mathbb{N}}

\usepackage{euler, epic, eepic}

\usepackage[mathscr]{euscript}
\DeclareSymbolFontAlphabet{\mathrsfs}{rsfs}

\numberwithin{thmcounter}{section}
\newaliascnt{thmauto}{thmcounter}

\newaliascnt{defauto}{thmcounter}

\newaliascnt{exauto}{thmcounter}

\newaliascnt{lemauto}{thmcounter}

\newaliascnt{propauto}{thmcounter}

\newaliascnt{corauto}{thmcounter}

\newaliascnt{remauto}{thmcounter}

\newaliascnt{notauto}{thmcounter}

\newaliascnt{conauto}{thmcounter}

\theoremstyle{definition} % me
\newtheorem{theorem}[thmauto]{Theorem}
\newtheorem{example}[exauto]{Example}
\newtheorem{lemma}[lemauto]{Lemma}
\newtheorem{proposition}[propauto]{Proposition}
\newtheorem{corollary}[corauto]{Corollary}
\newtheorem{definition}[defauto]{Definition}
\newtheorem{remark}[remauto]{Remark}

\includecomment{clean}
\includecomment{precise}
\excludecomment{sources}
\newcommand{\auth}[1]{}
\begin{sources}
\renewcommand{\auth}[1]{#1}
\end{sources}

\includecomment{questions-public}
\includecomment{questions-private}
\newcommand{\reply}[1]{}
\begin{questions-private}
\renewcommand{\reply}[1]{$\bigstar$ {\bf reply}: {\color{red}#1}}
\end{questions-private}

\includecomment{todo-private}
\newcommand{\todo}[1]{}
\begin{todo-private}
\renewcommand{\todo}[1]{$\dagger$ {\bf todo}: {\color{blue}#1}}
\end{todo-private}

\everymath{\displaystyle}
\allowdisplaybreaks
\title[From Trees to Barcodes and Back Again II]{From Trees to Barcodes and back again II: Combinatorial and probabilistic aspects of a topological inverse problem}
\author[J. Curry, J. DeSha, A. Garin, K. Hess, L. Kanari, B. Mallery]{Justin Curry, Jordan DeSha, Adélie Garin, Kathryn Hess, Lida Kanari, and  Brendan Mallery}
\date{\today}

\begin{document}
\maketitle

%\vspace{-.4in}

\begin{abstract}
In this paper we consider two aspects of the inverse problem of how to construct merge trees realizing a given barcode. Much of our investigation exploits a recently discovered connection between the symmetric group and barcodes in general position, based on the simple observation that death order is a permutation of birth order.
The first important outcome of our study is a clear combinatorial distinction between the space of phylogenetic trees (as defined by Billera, Holmes and Vogtmann) and the space of merge trees. Generic BHV trees on $n+1$ leaf nodes fall into $(2n-1)!!$ distinct strata, but the analogous number for merge trees is equal to the number of maximal chains in the lattice of partitions, i.e., $(n+1)!n!2^{-n}$.
The second aspect of our study is the derivation of precise formulas for the distribution of tree realization numbers (the number of merge trees realizing a given barcode) when we assume that barcodes are sampled using a uniform distribution on the symmetric group. We are able to characterize some of the higher moments of this distribution, thanks in part to a reformulation in terms of Dirichlet convolution. 
This characterization provides a type of null hypothesis, apparently different from the distributions observed in real neuron data and opens the door to doing more precise science.
\end{abstract}

%\vspace{-.1in}
\tableofcontents

\section{Introduction}

Trees have a nearly universal presence as a structure for organizing relationships between objects.
From hierarchical arrangements that are useful in the classification of species, to more immediate geometric applications in modeling neuron morphology \cite{TNS,TMD, TMD2}, trees have proved to be an indispensable tool.
However, as is natural for such a universal concept, subtle variations introduce important differences that are not always commented on.
In this paper, we are interested in the comparison of the notion of \emph{merge trees}, which is an important tool in topological data analysis (TDA), and that of \emph{metric phylogenetic trees}, which has gained a tremendous traction since its formalization by Billera, Holmes and Vogtmann \cite{BHV}, along with their combinatorial variants.
%%%Our motivation for comparing these objects comes from studying the combinatorics of certain inverse problems in TDA, which in turn is motivated by the study of neuroscience.
% \begin{itemize}
%     \item \emph{combinatorial trees}, which is the usual notion of tree used in, for instance, graph theory;
%     \item \emph{merge trees}, which arise from sublevel filtrations on topological spaces, and which we henceforth refer to as trees;
%     \item combinatorial classes of \emph{phylogenetic trees} which are obtained by forgetting precise edge lengths for the (metric) phylogenetic trees arising in evolutionary biology. (Metric) phylogenetic trees have recently received much attention in the Topological Data Analysis (TDA) community 
%     due to their pervasiveness in applications, their rich structure (cite Billera, Vogtmann and Holmes), and close resemblance to merge trees.
% \end{itemize} 

Our interest in delineating these objects comes in part from the fact that both merge trees and metric phylogenetic trees have associated \emph{barcodes}, which are topological invariants obtained from the persistent homology of a filtered space.
%Apart from delineating tree types, we are interested in relations between tree spaces and \emph{barcodes}, a topological invariant obtained from the persistent homology of a filtered space.
Since their introduction, barcodes or \emph{persistence diagrams} have become the standard topological summary used in TDA. 
Like all summaries, barcodes forget information about the space they are computed from. 
Thus, even when restricting to a specific set of topological spaces like trees, one may find that many different shapes give rise to the same barcode. 
Quantifying this failure of injectivity into a summary space is the realm of \textit{topological inverse problems}. 
Understanding such problems is crucial for comparing different representations of objects arising in both pure mathematics and in data science.

\subsection{High-Level Overview and Motivation}
In this paper we consider two aspects of the inverse problem of  constructing merge trees realizing a given barcode, motivated by recent work in neuroscience. In particular, the tools developed in \cite{TNS,TMD} have proven useful for the study of neuron morphologies \cite{TMD2}, which can be modeled by rooted trees, i.e., acyclic binary graphs with a distinguished vertex called the root (which corresponds to the neuron's soma), embedded in $\R^3$. In the terminology of this paper, this structure is most faithfully represented by merge trees. 

In \cite{TMD} the authors introduced the Topological Morphology Descriptor (TMD), an algorithm that returns a barcode from a tree, keeping track of the lengths of each branch with respect to a given filtration, but forgetting the adjacency relations between the branches. In this article we expand this investigation and systematically study the inverse problem from a combinatorial point of view.
We hope that understanding this relation will provide insight into the complex structures of neurons; see Figure \ref{tree_barcode_space} for a schematic. 

\begin{figure}[h]
    \centering
    \includegraphics[scale = 0.3]{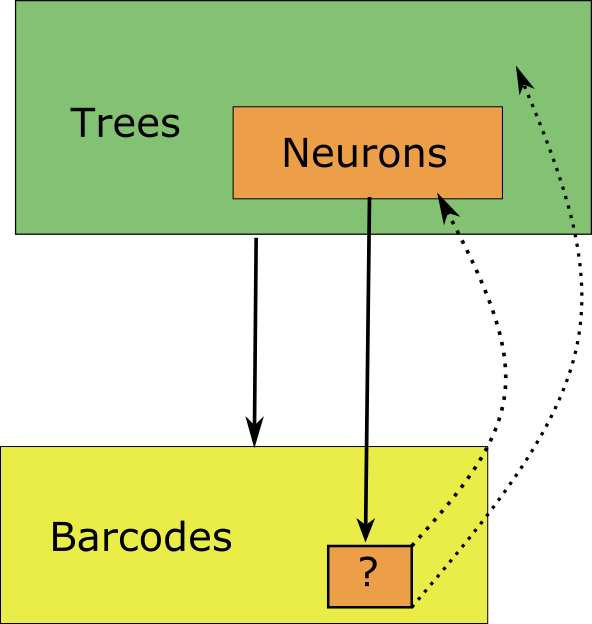}
    \caption{Motivation for understanding the pre-image of a barcode: Given a barcode computed from a neuron, what do all of its pre-images look like?}
    \label{tree_barcode_space}
\end{figure}

The general approach to the merge tree-to-barcode inverse problem is as follows. Any barcode can be realized by finitely many  trees, the number of which is called the tree realization number (TRN) or simply the realization number of the barcode. 
As observed in \cite{curry2017fiber} and \cite{TRN}, the realization number of a barcode in general position can be computed by certain containment relations between its bars, viewed as intervals on the real line. 
One of the crucial observations of \cite{TRN} is that these containment relations partition the set of barcodes (on $n$ bars) into equivalence classes, indexed by permutations in $S_n$, the symmetric group on $n$ letters. 
%This representation as a permutation allows one to easily compute the realization number of the barcode, and hence quantify the failure of injectivity, thereby giving rise to our inverse problem. 
The representation of a barcode by a permutation not only gives a formula for the tree realization number (Lemma \ref{lem:TRN-left-inversion}), but also opens the door to deeper connections between inverse problems in TDA, group theory, and combinatorics.

Besides quantifying the relative ``descriptive power'' of different summaries, in \cite{TRN} it was shown that the realization number could be used as a statistic to distinguish distributions of trees.
Figure~\ref{realization_bio} shows (log) realization numbers computed from different tree distributions, obtained by computing the realization number either from actual trees, such as neurons, or by randomly generating barcodes with specific properties. 
The datasets used were \emph{(i)} real neurons (basal and apical dendrites, drawn in red and purple), \emph{(ii)} random barcodes where the birth $b_i$ is picked, then the death $d_i$ is chosen to be larger than $b_i$, and \emph{(iii)} random barcodes with separated births and deaths so that the induced distribution on the symmetric group is uniform (see Section \ref{sec_dist_barcodes}). 
The results are striking: barcodes computed from neurons exhibit a very different distribution than barcodes with uniformly drawn permutation type; see Figure \ref{realization_bio} for a graphical comparison. 

\begin{figure}
    \centering
    \includegraphics[scale = 0.7]{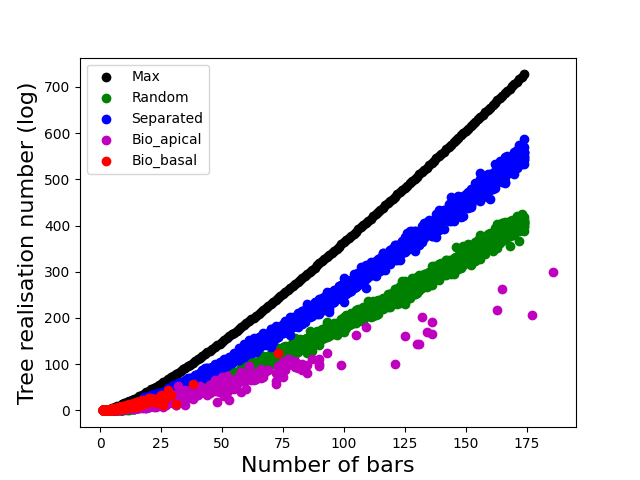}
    \caption{The log of the tree-realization number for barcodes with varying numbers of bars for TMD of basal dendrites (red), apical dendrites (purple) in comparison with ``random'' barcodes as defined in Section \ref{sec_dist_barcodes} (green), barcodes with separated births and deaths such that the distribution induced on the symmetric group is uniform (blue, see section \ref{sec_dist_barcodes} and Proposition \ref{expected_real_number}), and the maximum tree-realization number ($n!$ for $n+1$ bars) (black). }
    \label{realization_bio}
\end{figure}

In this paper, we study the realization numbers computed from barcodes with uniform permutation type (i.e., drawn from the uniform distribution on the symmetric group). 
We view this as essential for the realization number to be used for applications, as it establishes a fundamental null hypothesis for the invariant. 
Our tools are mainly combinatorial, leading us to discover unexpected connections between the inverse problem and other classical combinatorial objects. One of our main theorems (Theorem \ref{MaxChain}) casts the classic result of Erd\H{o}s that counts the number of maximal chains in the lattice of set partitions in a new, merge-tree light.
It was this result that not only permitted an easy calculation of the expected tree realization number, but also further established the fundamental differences between combinatorial classes of merge trees and phylogenetic trees.
We now provide a more detailed overview of the paper.
%We hope to initiate a deeper study of this connection in a forthcoming paper. 
% Lastly, we take a small detour in notions of average of merge trees developed in \cite{structural}. 
% We establish statistical results on how the average merge tree obtained from a barcode behaves as an invariant. We show with two simple examples that the average merge trees defined here are not compatible with reasonable notions of average barcode, opening the door to a wealth of other questions about appropriate statistics that carry over between tree spaces and barcode space. Obtaining such compatible statistics could allow scientists to share their tools between the different spaces.

 %\medskip
 
%\subsection{Related work}

\subsection{Detailed Overview}

After the introduction, we start in earnest by reviewing the basic properties of trees and barcodes in Section \ref{sec_trees_barcodes}. 
The basic graph-theoretic notion of a tree is reviewed in Definition \ref{top_tree}, as are the notions of labelling and isomorphism. 
Labellings offer one important way of distinguishing merge trees (Definition \ref{MT}) and metric phylogenetic trees (Definition \ref{defn:phylo-trees}), but Proposition \ref{prop:BHV-MT-map-compare} provides a more carefully stated distinction between the notions of BHV space, labelled merge tree space, and merge tree space. In this first subsection, combinatorial notions of merge trees and phylogenetic trees are also introduced.
The pertinence of these combinatorial notions becomes evident after we introduce barcodes in Section \ref{sec:barcodes}, which allows us to review the inverse problem for merge trees in Section \ref{sec_rel_trees_barcodes}, where the combinatorial (permutation) type of the barcode is all that matters (see Section \ref{merge_tree_sym_group}).

Section \ref{sec:combo-alg-TRN} marks the beginning of this paper's contribution to the literature.
In Section \ref{sec_real_number}, we formalize the observation of \cite{TRN} that the tree realization number (TRN) is a function of the symmetric group, by expressing the TRN in terms of the left-inversion vector associated to a permutation.
We take a minor detour in Section \ref{sec_continuous_path} to observe that the combinatorial equivalence class of each barcode is convex (Lemma \ref{lemma_cont_path}), which is of use later when we choose certain standard forms for barcodes (Definition \ref{defn:standard-form-BC}) and merge trees (Definition \ref{defn:standard-form-MT}).
We continue the algebraic analysis of the TRN in Section \ref{sec:Bruhat}, where we prove that when the symmetric group is equipped with a certain partial order (Definition \ref{defn:left-Bruhat}), the TRN is an order-preserving map.
After proving that every pair of combinatorially equivalent merge trees can be connected by a line of merge trees (Lemma \ref{lem:line-of-MTs}), we show that the sum of the tree realization numbers is equal to the total number of combinatorial types of merge trees in Lemma \ref{lem:count-combo-classes-MTs}.
Theorem \ref{MaxChain} in turn states that this number is equal to the number of maximal chains in the lattice of partitions (Definition \ref{defn:lattice-of-partitions}), which is $(n+1)!n!2^{-n}$.
This result provides a stark combinatorial contrast with the well-known fact that there are $(2n-1)!!$ types of labelled binary trees on $n+1$ nodes \cite{NumberTrees}. Section \ref{sec_MT_vs_PT} explores this contrast in greater depth by making quantitative the observation that whereas merge trees fiber over the symmetric group in a nice way, phylogenetic trees do not.

Section \ref{sec_statistics_real} finally delivers closed-form formulas for some of the trend lines in Figure \ref{realization_bio}.
We cover briefly two methods to generate random barcodes in Section \ref{sec_dist_barcodes}, before characterizing the distribution of tree realization numbers (when sampled uniformly on the symmetric group) in terms of Dirichlet convolution in Theorem \ref{thm:Dirichlet-distribution}.
The paper concludes with Proposition \ref{prop:expected-log-TRN}, which uses the left-inversion vector representation of the TRN to give a closed formula for the expected log realization number.

\subsection{Related Work}
This paper touches on many classical concepts related to trees and combinatorics, so providing a complete list of related work is impossible. However, the literature on inverse problems for TDA can be reviewed briefly here.

The concept of a geometric realization of a persistence module was considered in~\cite{lesnick2015theory} in order to prove a universality result for the interleaving distance. 
In \cite{gameiro2016continuation} the authors initiated an algorithmic study of how to find a point cloud that realizes a given persistence diagram.
While these articles are concerned with finding single realizations of persistent signatures, the present article focuses on the study of the entire pre-image of the persistent homology pipeline.

In the same vein, there is \cite{curry2017fiber}, which focused on the setting of functions on the interval and their associated merge trees.
Some of the results there were independently rediscovered and extended in \cite{TRN}, which inspired the present collaboration.
Both \cite{cyranka2020contractibility} and \cite{leygonie2021fiber} are more recent articles that investigate the fiber of the persistence map in settings that are different from ours. 

We note that the study of the (non-) injectivity of certain topological transforms is also an aspect of topological inverse problems, see \cite{oudot2017barcode,ghrist2018persistent,curry2018many,maria2019intrinsic, solomon2021geometry} for a sampling of these articles and \cite{oudot2020inverse} for a recent survey.
Better understanding the precise failure of injectivity of certain TDA invariants led to the development of enriched topological summaries (ETS) that remediate these failures, opening a promising line of research; see \cite{catanzaro2020moduli} and \cite{curry2021decorated} for some examples of these ETS.

Section \ref{sec:Bruhat} of this paper explores the relationship between the Bruhat order on the symmetric group and barcode equivalence classes. A similar connection was observed in \cite{master_thesis} in a different context.

%\subsection{Acknowledgements} moved to funding statement

%JC would like to acknowledge NSF Grant
%CCF-1850052 and NASA Contract 80GRC020C0016 for supporting his research.

%AG should acknowledge Sinergia funding.

\section{Background on Trees and Barcodes}\label{sec_trees_barcodes}
In this section, we assume basic familiarity with persistent homology in degree $0$, even though it is not necessary to understand persistence for most of these definitions. 
For a more algorithmic review of the topic in the case of trees, see \cite{TRN}. We begin by reviewing the necessary background and combinatorial results from \cite{curry2017fiber} and \cite{TRN}.
Most of this section reviews prior work, though Proposition \ref{prop:BHV-MT-map-compare} provides a novel comparison of merge trees and phylogenetic trees and foreshadows results later in the paper.

\subsection{Trees, Merge Trees and Phylogenetic Trees}

There are many notions of trees in mathematics and the sciences. We review a few of these here and explain their differences. We start with the simplest definition, that of a combinatorial tree. 

\begin{definition}\label{top_tree}
A \emph{combinatorial tree} $T$ is a connected, acyclic, binary graph. It is \emph{finite} if the number of vertices is finite. A \emph{rooted tree} is a combinatorial tree with a distinguished vertex of degree $1$ called the \emph{root}. Non-root vertices of degree $1$ are called \emph{leaves}. 

%A combinatorial tree equipped with an embedding into $\R^3$ is called a \emph{geometric tree}. 

A \emph{labelling} of a combinatorial tree $T$ is a bijective map from its set of vertices $V(T)$ to a set $S$ of labels. 
A labelling is \emph{ordered} if $S$ is a subset of of the natural numbers $\mathbb{N}$. An ordered labelling of a tree with $n$ vertices gives rise to an $n\times n$ \emph{adjacency matrix}, of which the $(i,j)$-coefficient is $1$ if there is an edge between the vertices labelled $i$ and $j$ and is $0$ otherwise. 

Two combinatorial trees $T$ and $T'$ are \emph{isomorphic} if there is a bijective map $T \to T'$ that sends vertices to vertices in an adjacency-preserving way: if two vertices in $T$ are connected by an edge, then so are their images. 
Equivalently, $T$ and $T'$ are isomorphic if there exist ordered labellings of both with respect to which their adjacency matrices are identical.
\end{definition}

In this paper, we assume all trees are finite. Moreover, we assume that there are no vertices of degree $2$, that is, each vertex is either a \emph{bifurcation} or \emph{branching} point, i.e.,~a vertex of degree $3$, or a \emph{termination}, i.e.,~a vertex of degree $1$, such as the leaf nodes or the root. 

When rooted trees are considered, there is a natural way to induce an orientation on the edges of the tree: for each vertex $v$, there is a unique path from $v$ to the root $r$. Every edge of the tree is oriented from the vertex further from $r$ to the closer one (with respect to the graph path distance). 
A vertex $v$ of $T$ is a \emph{parent} of a vertex $w$ if there is a directed edge from $w$ to $v$; the vertex $w$ is then a \emph{child} of $v$. 
Each vertex of $T$ has a unique parent, except for the root $r$, which has no parent at all.
Note that a finite combinatorial tree $T$ is fully specified by its set of vertices, equipped with the partial order specified by the ``is a parent of'' relation. 
The language of ``parents'' and ``children'' obviously comes from studying ancestral relations for people (as in family trees) and species (as in phylogenetic trees).
There are also situations where the parent-child relation is determined in part by a notion of ``height,'' which is how merge trees are defined.

\begin{definition}\label{MT}
A \emph{merge tree} is a rooted combinatorial tree $T$, together with a function on the vertices
$h:V(T)\longrightarrow\R\cup\{\infty\}$, called a \emph{height function}, that satisfies two properties.
\begin{enumerate}
    \item If $v$ is the parent of $w$, then $h(v)\geq h(w)$.
    \item If $r$ is the root node, then $h(r)=\infty$.
\end{enumerate}
Two merge trees $(T,h)$ and $(T',h')$ are \emph{isomorphic} if there is a graph isomorphism $\varphi:T\to T'$ that preserves heights, i.e., $h=h' \circ \varphi$.
A \emph{generic merge tree} is a merge tree $(T,h)$ such that the height function $h:V(T)\to \R$ is injective.
We always assume our merge trees are generic, unless otherwise indicated.
\end{definition}

\begin{remark}[Drawing Conventions for Merge Trees]\label{rmk:draw-finite}
Many authors choose to draw merge trees so that the function $h:V(T)\to\R$ resembles height when embedded in the page.
This has the effect of placing the root node  higher than the leaf nodes, contrary to how trees appear in nature.
To honor the natural orientation and size of trees in nature, we draw our merge trees with the opposite convention, so that the root is lower than the leaves and so that $f(r)=\infty$ is represented with a finite value $N$.
%This may make some ordering assumptions slightly confusing, so the reader is cautioned to take care.
\end{remark}

\begin{remark}[Alternative Definition of Merge Trees]\label{alternative_mt}
Another, perhaps more common, definition of a merge tree is that it is the Reeb graph of the epigraph of a function. 
From this point of view, the merge tree $T$ of a real-valued function $f: X \to \mathbb{R}$ is the quotient space of the epigraph $\Gamma^+:=\{(x,t)\in X\times \R \mid f(x)\leq t\}$ by the equivalence relation specified by $(x,t) \sim (y,s)$ if and only if $s=t$ and $x$ and $y$ are in the same path component of the sublevel set filtration of $f$ at $t$, i.e.,
$[x]=[y] \in \pi_0(f^{-1}(-\infty, t])$. 
Since the projection map from $\Gamma^+$ onto the second coordinate is constant on equivalence classes, this projection map factors to define the height function.
Under reasonable tameness conditions, the quotient space is homeomorphic to the geometric realization of a combinatorial tree, where vertices correspond to connected components of ``critical'' points.
\end{remark}

\begin{example} 
A typical example of merge tree is one arising from measuring height on an embedded manifold $X\subseteq \R^n$.
Here ``height'' can be thought of as the scalar product with a specified unit vector.
Figure \ref{cactus} shows a simple example of a topological space and the corresponding merge tree. 

\begin{figure}
    \centering
    \includegraphics[scale=0.4]{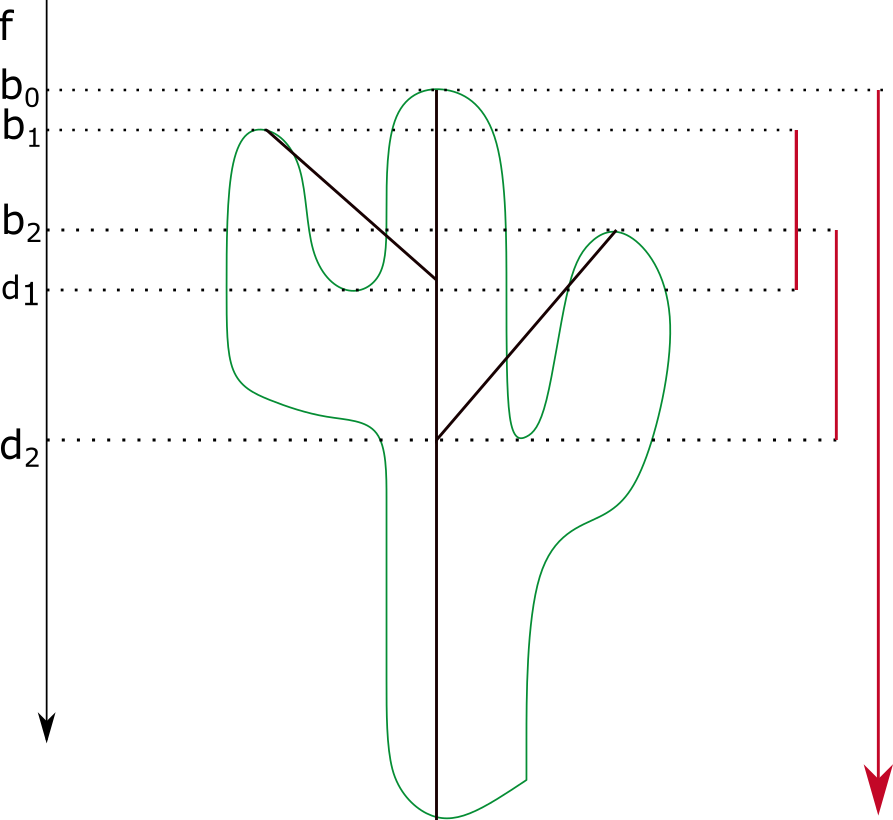}
    \caption{A circle $X$ is embedded in $\R^2$ and drawn in green to resemble a cactus with the height function $f$ measuring distance down the page. The corresponding merge tree (Definition \ref{MT}) is drawn in black. The barcode of the persistence module in degree $0$ (Definition \ref{def_pers_mod}) associated to $(X,f)$ is shown in red on the right.}
    \label{cactus}
\end{figure}
\end{example}

There is a natural ordered labelling on the vertices of a generic merge tree $(T,h)$, inherited from the function $h$, by ordering the vertices according to their $h$-value: the leaf node with lowest $h$-value is labelled $0$, and the remaining nodes are labelled based thereafter on the order in which they appear. 
We call the labels on the leaves the \emph{birth labels} and the ones on the internal vertices the \emph{death labels}, for reasons that will become clear later in the paper when we review persistent homology.
\\

We are now in a position to state the first novel definition of the paper.
Recall that two graphs are isomorphic if they admit ordered labellings making their adjacency matrices the same.
A merge tree includes the additional data of heights of each node.
By focusing separately on the order of births and the order of deaths, along with adjacency data, we have a more flexible notion of a merge tree.

\begin{figure}
    \centering
    \includegraphics[width=.6\textwidth]{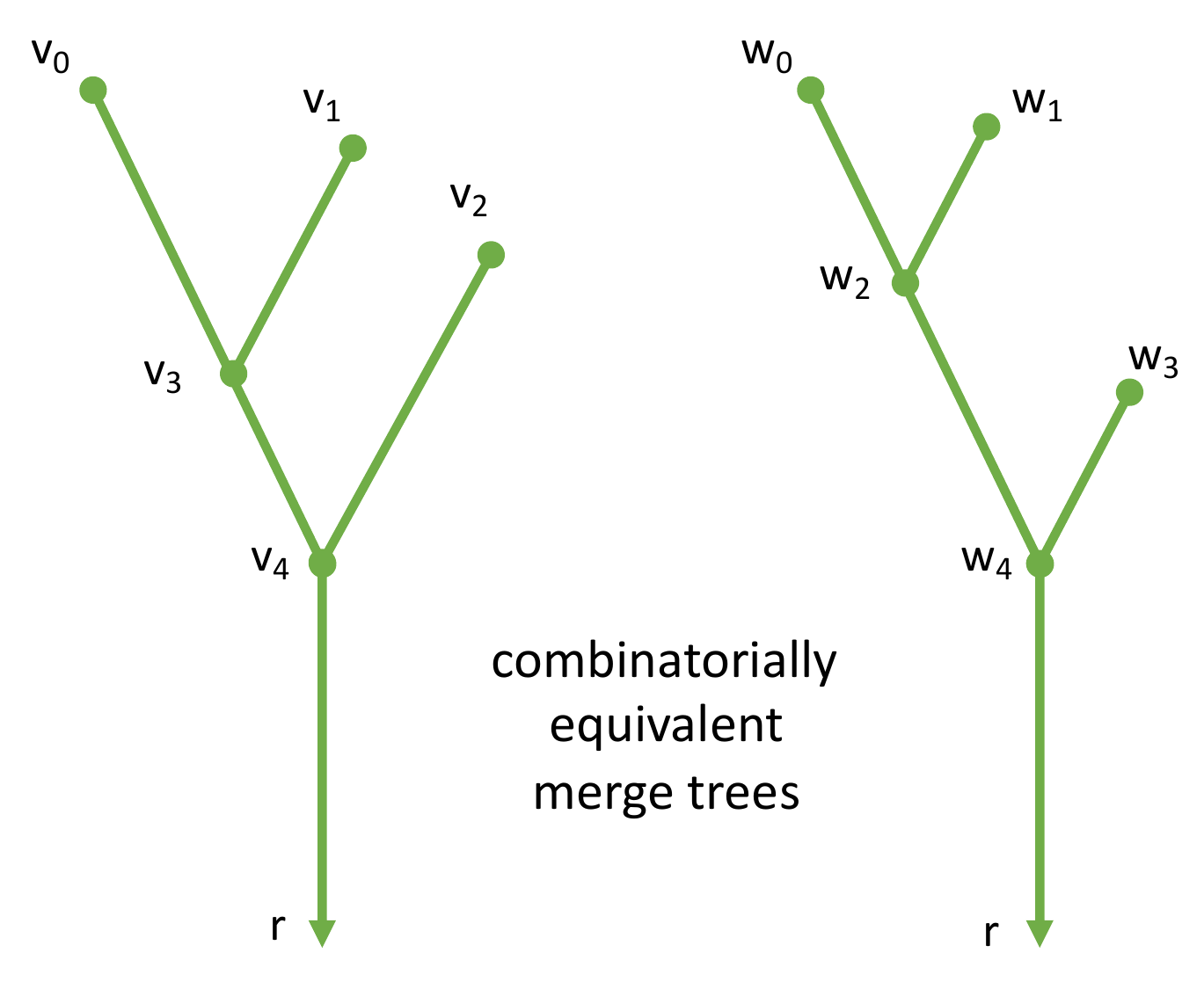}
    \caption{Two combinatorially equivalent merge trees are shown. Notice that the total order of the vertices is not preserved, but the orders among leaf nodes and internal nodes are preserved separately.}
    \label{fig:combo-equiv-MTs}
\end{figure}

\begin{definition}\label{defn:combo-merge-tree}
Two generic merge trees $(T,h)$ and $(T',h')$ are \emph{combinatorially equivalent} if they are isomorphic as graphs via a graph isomorphism preserving the orders of births and of deaths, respectively.
In more detail, $(T,h)$ and $(T',h')$ are combinatorially equivalent if there exists a graph isomorphism $\varphi:T\to T'$ such that the following conditions hold.
\begin{enumerate}
    \item For every pair of leaf (birth) nodes $v_i$ and $v_j$ in $T$, if $h(v_i)<h(v_j)$, then  $h'(\varphi(v_i))<h'(\varphi(v_j))$. 
    \item For every pair of internal (death) nodes $v_i$ and $v_j$ in $T$, if $h(v_i)<h(v_j)$, then $h'(\varphi(v_i))<h'(\varphi(v_j))$.
\end{enumerate}
We note that these two conditions specify two different sets for the logical quantifier and that the total order on vertices need not be preserved;  see Figure \ref{fig:combo-equiv-MTs} for an example.
\end{definition}

\begin{remark}\label{rmk_comb_MT}
Note that that combinatorial equivalence classes of merge trees are simply combinatorial trees equipped with a labelling of the leaves and a labelling of the internal nodes. We call such a tree a \emph{combinatorial merge tree}.
\end{remark}

\begin{example}[Translation Invariant]
Consider two generic merge trees $(T,h)$ and  $(T,h')$ such that $h'=h+\Delta$ for some real number $\Delta$.
We say $(T',h')$ is a \emph{translation} of $T$.
A generic merge tree is combinatorially equivalent to any translation of itself.
However, combinatorial equivalence detects relationships more general than translation; see Figure~\ref{fig:combo-equiv-MTs}.
\end{example}

\begin{example}[Sensitivity to Generators]
Although the two merge trees in Figure~\ref{fig:merge-tree-split-topology} are isomorphic as graphs, the only possible graph isomorphism reverses the birth order, hence these generic merge trees are not combinatorially equivalent. 
Notice that the homology generator of the essential class (see Section \ref{sec:barcodes}) starts with the node labelled by $0$ or $A$ on the left hand side, while on the right hand side, it starts with the $0$ or $B$ label. This is sometimes called ``instability'' or ``sensitivity'' of generators in TDA. Together with Figure~\ref{fig:combo-equiv-MTs}, these specify the three possible combinatorial equivalence classes of merge trees with three leaf nodes.

\begin{figure}
    \centering
    \includegraphics[width=.8\textwidth]{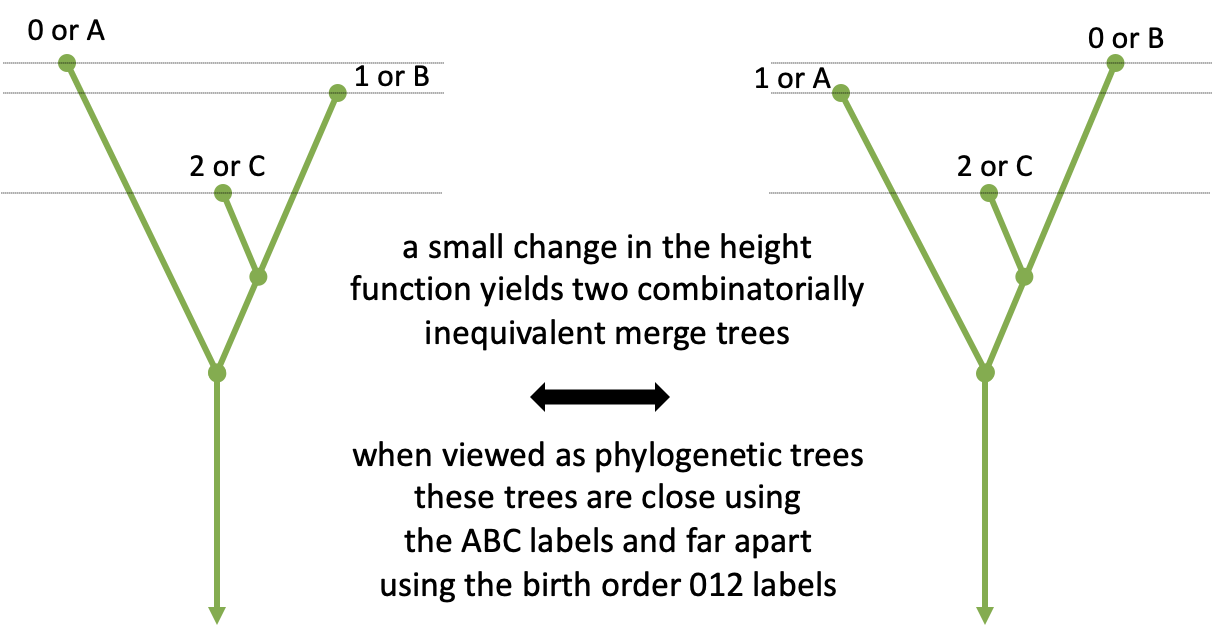}
    \caption{Two generic merge trees that are isomorphic as graphs. When they are regarded as phylogenetic trees we fix alphabetical (`ABC') names for the leaf nodes, as if the nodes represented species that went extinct at different times. With this labelling they are considered close in the metric defined by \cite{BHV}. When these are regarded as merge trees they are naturally unlabelled and are close in the interleaving distance \cite{interleaving}, but if we use birth order (`012') to label the leaf nodes and regard them as phylogenetic trees then they are far apart; see Proposition \ref{prop:BHV-MT-map-compare}.}
    \label{fig:merge-tree-split-topology}
\end{figure}

\end{example}

As mentioned earlier, most of the language concerning trees is inspired by the study of ancestral relationships. 
Although trees have been used for this purpose for centuries, a formal definition of a phylogenetic tree---and more importantly a clear coordinatization on the set of all phylogenetic trees---was given only somewhat recently in the landmark paper of Billera, Holmes and Vogtmann \cite{BHV}. 
%A phylogenetic tree encodes the relation between species in the evolution process. 
We review some of these definitions, modifying the terminology slightly for our purposes.

\begin{definition}\label{defn:phylo-trees}
%A \emph{rooted (combinatorial) phylogenetic tree} is a rooted combinatorial tree $T$ with a labelling on the leaves.
A \emph{metric phylogenetic tree} is a rooted combinatorial tree $T$ endowed with
\begin{enumerate}
    \item a labelling of the leaf nodes, and
    \item a non-negative real number associated to every parent-child pair.
\end{enumerate}
The values assigned to each parent-child pair can be considered as weights on the graph edges. 
By contrast, a \emph{combinatorial phylogenetic tree} is a rooted combinatorial tree with just a labelling of the leaf nodes.
If we say \emph{phylogenetic tree} without any modifier, we always mean a combinatorial phylogenetic tree.
\end{definition}

\begin{example}
Figure \ref{cayley_averagetree}B shows all combinatorial classes of merge trees with four leaves and Figure \ref{cayley_averagetree}C shows all combinatorial classes of phylogenetic trees with four leaves.
\end{example}

One of the key differences between metric phylogenetic trees and merge trees is that phylogenetic trees always have labelled leaf nodes, with labels independent of the lengths on the edges. This makes sense because \emph{BHV space}---the set of all possible metric phylogenetic trees on $n$ leaf nodes, denoted $\mathcal{MPT}_n$---documents all possible evolutionary relationships among $n$ fixed species.
The labels matter because the involved species matter.

On the other hand, the set of all merge trees with $n$ leaf nodes, written $\mathcal{MT}_n$, consists of isomorphism classes of merge trees, see Definition \ref{MT}.
  We consider also the set $\mathcal{LMT}_n$ of labelled merge trees with $n$ leaves, where the labelling is arbitrary (see Definition \ref{top_tree}).  Let
\[
\mathcal{I}:\mathcal{LMT}_n \longrightarrow \mathcal{MT}_n,
\]
denote the map
that sends a labelled merge tree to its isomorphism class.

We describe the relationship between these two types of tree spaces in the following proposition.

\begin{proposition}\label{prop:BHV-MT-map-compare}
For every $\Delta\in \R$, there is an injective map from the set of metric phylogenetic trees with $n$ leaves, $\mathcal{MPT}_n$, to the set of labelled merge trees with $n$ leaves, $\mathcal{LMT}_n$
\[
\mathcal{H}_{\Delta}:\mathcal{MPT}_n \longrightarrow \mathcal{LMT}_n.
\]
such that
the composite $\mathcal{I}\circ\mathcal{H}_{\Delta}$ has a fiber of cardinality $n!$ over generic merge trees, corresponding to permutations of the labels on the leaf nodes.
Moreover, if $\Delta\geq 0$, there is a natural map
\[
\mathcal{T}_{\Delta}:\mathcal{MT}_{n,\text{generic}} \longrightarrow \mathcal{MPT}_n
\]
that sends a generic merge tree to a metric phylogenetic tree that is labelled by birth order and where the distance from the root node to its child is $\Delta$.
\end{proposition}
\begin{proof}
Given a metric phylogenetic structure on a rooted tree $T$, we can define a height function $h$ on $T$ as follows.  Every node $v$ that is not the root node $r$ is assigned the function value $h(v):=\Delta-d(r,v)$, where $d$ is the sum of the weights of each edge along the unique path connecting $r$ to $v$.
This defines the map $\mathcal{H}_{\Delta}$ in the statement of the proposition.

As explained earlier, every generic merge tree admits  a canonical ordering of its leaf nodes by height order.
If two generic labelled merge trees in the image of $\mathcal{H}_{\Delta}$ are isomorphic as merge trees, then there is a unique permutation of the $n$ leaf labels taking one labelling to the other. This proves the second statement.

Finally, the map $\mathcal{T}_{\Delta}$ sends a generic unlabelled merge tree $(T,h)$ to the metric phylogenetic structure on $T$ that has labels given by birth order and where the weight on an edge is given by the difference in heights of its two vertices. The distance from the root node to its child is given by $\Delta$.
%The fact that this map is discontinuous can be seen from Figure \ref{fig:merge-tree-split-topology}.
\end{proof}
\begin{remark}
Each of the three sets above can be equipped with topologies.
In \cite{BHV}, the space of phylogenetic trees is topologized as a CAT(0) space where each orthant records a distinct \emph{split topology}. Both labelled merge trees and merge trees can be topologized using versions of the interleaving distance~\cite{structural}. Unfortunately, the map $\mathcal{T}_{\Delta}$ is discontinuous with respect to these topologies, as can be seen from Figure \ref{fig:merge-tree-split-topology}.

Proposition \ref{prop:BHV-MT-map-compare} shows that, despite their apparent similarity, there are significant differences between metric phylogenetic trees and merge trees.
Indeed neither of the maps above is a bijection.
However, if one quotients the set of labelled merge trees by translations, then the map induced by $\mathcal{H}_{\Delta}$ should be a bijection; alternatively one could modify the definition of merge trees so that the root node has a fixed height $N$, as in the drawing convention of Remark \ref{rmk:draw-finite}.
\end{remark}

Although the proposition and remark above identify certain differences and similarities between metric phylogenetic trees and merge trees, for this paper the most important distinction is in terms of combinatorial type. 
In this respect, merge trees and phylogenetic trees are distinguished by the explicit ordering of birth and death nodes. 
This observation will lead to different formulas for the numbers of top-dimensional strata in the set of phylogenetic trees $\mathcal{PT}_n$, which is $(2n-3)!!$, and in $\mathcal{MT}_n$, which is $(n-1)!n!2^{-n+1}$.
For now, however, the reader is encouraged to consult Table \ref{table_characteristic} and Figure \ref{fig:elderrule} for two convenient summaries of the similarities and differences between combinatorial trees, merge trees, (combinatorial) phylogenetic trees, and barcodes.

\begin{figure}
    \centering
    \includegraphics[width=\textwidth]{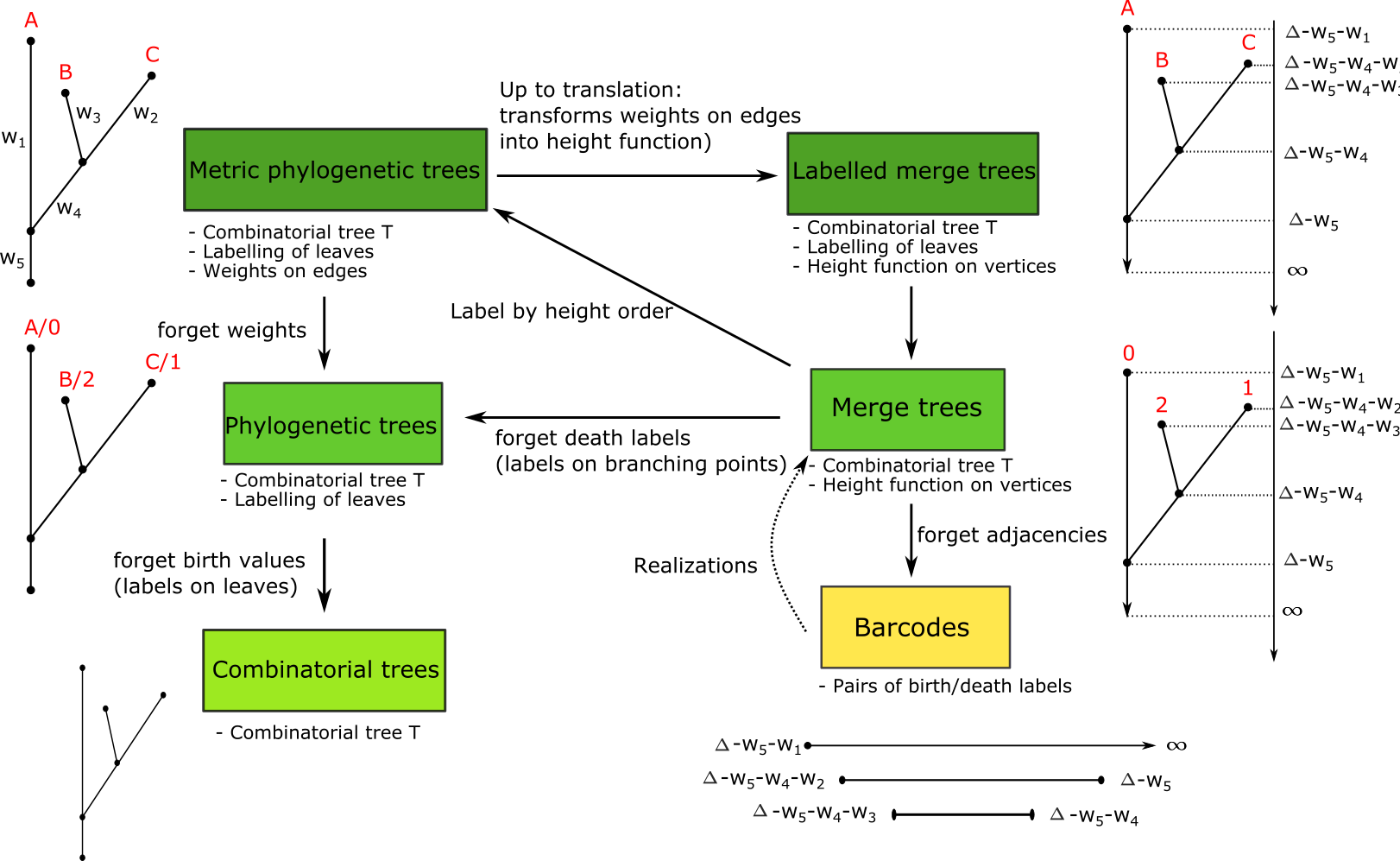}
    \caption{Summary of the different notions studied in this paper and their relations, as expressed in part by Proposition \ref{prop:BHV-MT-map-compare}. One can turn a metric phylogenetic tree (with labels A,B,C in red) into a labelled merge tree. Generic merge trees can be turned into metric phylogenetic trees by labelling according to birth order (labels $0,1,2$ in red), but this introduces a discontinuity. %For this specific example, we choose to assign the value $N$ to the root instead of $\infty$. 
    }
    \label{fig:elderrule}
\end{figure}

\subsection{Barcodes}\label{sec:barcodes}

We now recall the notions of persistent homology and barcodes. For reasons of brevity, we choose to use the categorical definition of persistent homology, but the reader who would like a more algorithmic version for the case of merge trees can read \cite{TRN} or the summary in Example \ref{barcode_comp_practice}.

\begin{definition}\label{def_pers_mod}
A \textit{persistence module} is a functor 
\[ F: (\mathbb{R}, \leq) \to Vect \]
where $(\mathbb{R}, \leq)$ is the real line with its total ordering $\leq.$
An \emph{interval module} is a persistence module $\Bbbk_I$ that is rank 1 on an interval $I\subseteq \R$ with identity maps internal to $I$ and $0$ elsewhere.
\end{definition}

A function $f: X \to \mathbb{R}$ is said to be \emph{tame} if the homology groups of the sublevel sets $\{f^t\}_{t \in \R}=\{f^{-1}((-\infty,t])\}_{t\in \R}$ have finite rank and change at a finite number of points. 

Tame functions $f$ have finitely many critical values $a_0,\cdots, a_n$, and the sublevel sets $f^{t_1}, f^{t_2}$ are homeomorphic when $t_1, t_2 \in (a_i, a_{i + 1})$ for $i = 0, \cdots, n - 1$. By \cite{crawley2015decomposition} we have the following decomposition theorem.

\begin{theorem}(Crawley-Boevey)
Any pointwise finite dimensional persistence module is isomorphic to a direct sum of interval modules, and this decomposition is unique up to reordering.
\end{theorem}

%%Barcode (MT barcode)
\begin{definition}
Let $F$ be a persistence module with decomposition $F\cong \bigoplus _{j\in \mathcal J} \Bbbk_{I_j}^{\oplus n_j}$. The \emph{barcode} of $F$ is the multiset
$$B(F) = \{(I_j, n_j)\}_{j \in \mathcal{J}}.$$ 
In most applications, each interval $I_j$ is of the form $[b_j,d_j)$, where $b_j$ is the \emph{birth} of the homological feature corresponding to $I_j$ and $d_j$ its \emph{death}. We call the interval $[b_j,d_j)$ a \emph{bar} in the barcode $B$. 
\end{definition}

In this paper, we represent barcodes graphically by drawing the interval between $b_j$ and $d_j$ for each index $j$. Sometimes barcodes are represented by \emph{persistence diagrams}, i.e., sets of points in $\R^2$ where the $x$-coordinate indicates birth time and the $y$-coordinate death time. Note that $x\leq y$ always in this representation.

\begin{example}[Barcodes for Merge Trees and the Elder Rule]\label{barcode_comp_practice}
Let $(T,h)$ be a merge tree.
Regarding $T$ as a one-dimensional simplicial complex, we can linearly interpolate the height function from the vertices to the entire tree.
The \emph{barcode of the merge tree $(T,h)$} is the barcode corresponding to the persistence module 
\[ 
F: (\mathbb{R}, \leq) \to Vect \qquad \text{where} \qquad F(t) = H_0\Big(h^{-1}\big((- \infty, t]\big)\Big). 
\] 
Although the barcode of $F$ is guaranteed to exist by virtue of Crawley-Boevey's theorem, there is a more direct way of constructing the barcode in the special case of merge trees, called the \emph{Elder rule} \cite{curry2017fiber}.

The Elder rule provides a concrete way to compute the barcode of a merge tree via decomposition into branches, i.e., each bar in the barcode corresponds either to a single edge or a list of adjacent edges in the merge tree.
According to the Elder rule, each leaf node marks the beginning of a bar in the barcode at the height of the leaf node.
If two leaf nodes $l_i$ and $l_j$ such that $f(l_i)>f(l_j)$ share a parent at vertex $k$, the branch that was born ``earlier" at $l_j$  survives as it is ``elder'', and the branch born $l_i$ dies, creating a bar $[f(l_i),f(k))$ in the barcode.

Under this rule, every bar begins at a leaf node and ends at an internal node with the sole exception of the bar that is born at the  leaf node with the lowest height, which is paired with infinity.
However, in our figures, in keeping with Remark \ref{rmk:draw-finite}, the lowest leaf node will be paired with $N=f(r)$, which is the height of the root node when viewed as an embedded finite tree.
A simple example is illustrated in Figure \ref{cactus}.
%; recall our drawing convention has flipped what is ``older'' or ``lower'' nodes are drawn higher on the page. 
\end{example}
 
 Although in general the barcode can be a true multiset, in this article we are concerned primarily with  barcodes that are actually sets, leading us to formulate the following definition.
 
\begin{definition}
A a barcode $B$ is \emph{strict} if is composed of  one half-infinite bar $[b_0, \infty)$, and a finite number of half open bars $[b_1, d_1),\cdots, [b_n, d_n) $ such that $b_0 < \cdots < b_n$ and $d_i\not=d_j$ if $i\not=j$. We refer to the half-infinite bar as \textit{essential}.
\end{definition}
 
\begin{example}
The barcode of a generic merge tree is always strict.
\end{example}

We summarise the different characteristics of combinatorial trees, merge trees, phylogenetic trees, and barcodes in Table \ref{table_characteristic}.

\begin{table}
\begin{tabular}{|l|l|c|c|c|}
\hline
\textbf{}                                                                               & \textbf{Combinatorial trees} & \multicolumn{1}{l|}{\textbf{Merge trees}} & \multicolumn{1}{l|}{\textbf{Phylogenetic trees}} & \multicolumn{1}{l|}{\textbf{Barcodes}} \\ \hline
\textbf{Height function}                                                                &                            & X                                         & \multicolumn{1}{l|}{}                            & \multicolumn{1}{l|}{}                  \\ \hline
\textbf{\begin{tabular}[c]{@{}l@{}}Label on leaves\\ (births)\end{tabular}}             &                            & X*                                        & X                                                & X*                                     \\ \hline
\textbf{\begin{tabular}[c]{@{}l@{}}Label on internal \\ vertices (deaths)\end{tabular}} &                            & X*                                        &                                                  & X*                                     \\ \hline
\textbf{Adjacency}                                                                      & \multicolumn{1}{c|}{X}     & X                                         & X                                                &                                        \\ \hline
\end{tabular}
\caption{Table summarising the attributes of each object defined above. Labels on leaves and internal vertices of merge trees are marked by an asterix to indicate that they are inherited from the height function. Similarly, the ``labels'' on barcodes (their birth and death values) are inherited from the height function on the tree.}
\label{table_characteristic}
\end{table}

\subsection{Realizations of Barcodes} \label{sec_rel_trees_barcodes}

As described in the previous section, every merge tree has an associated barcode.
It is natural to ask whether the map from merge trees to barcodes determined by the Elder rule is injective, but it is not hard to see that it is not.  
A somewhat more surprising result, proven independently in \cite{curry2017fiber} and \cite{TRN}, is that the failure of injectivity of the Elder rule map can be quantified for generic barcodes.
More precisely, we say that a merge tree $(T,h)$ \textit{realizes} a barcode $B$ if the barcode of $(T,h)$ is $B$.
The \textit{tree realization number}, $R(B)$, of a strict barcode $B$ is the number of combinatorial trees $T$ admitting a height function $h$ such that $(T,h)$ realizes $B$.

\begin{figure}
    \centering
    \includegraphics[width= \textwidth]{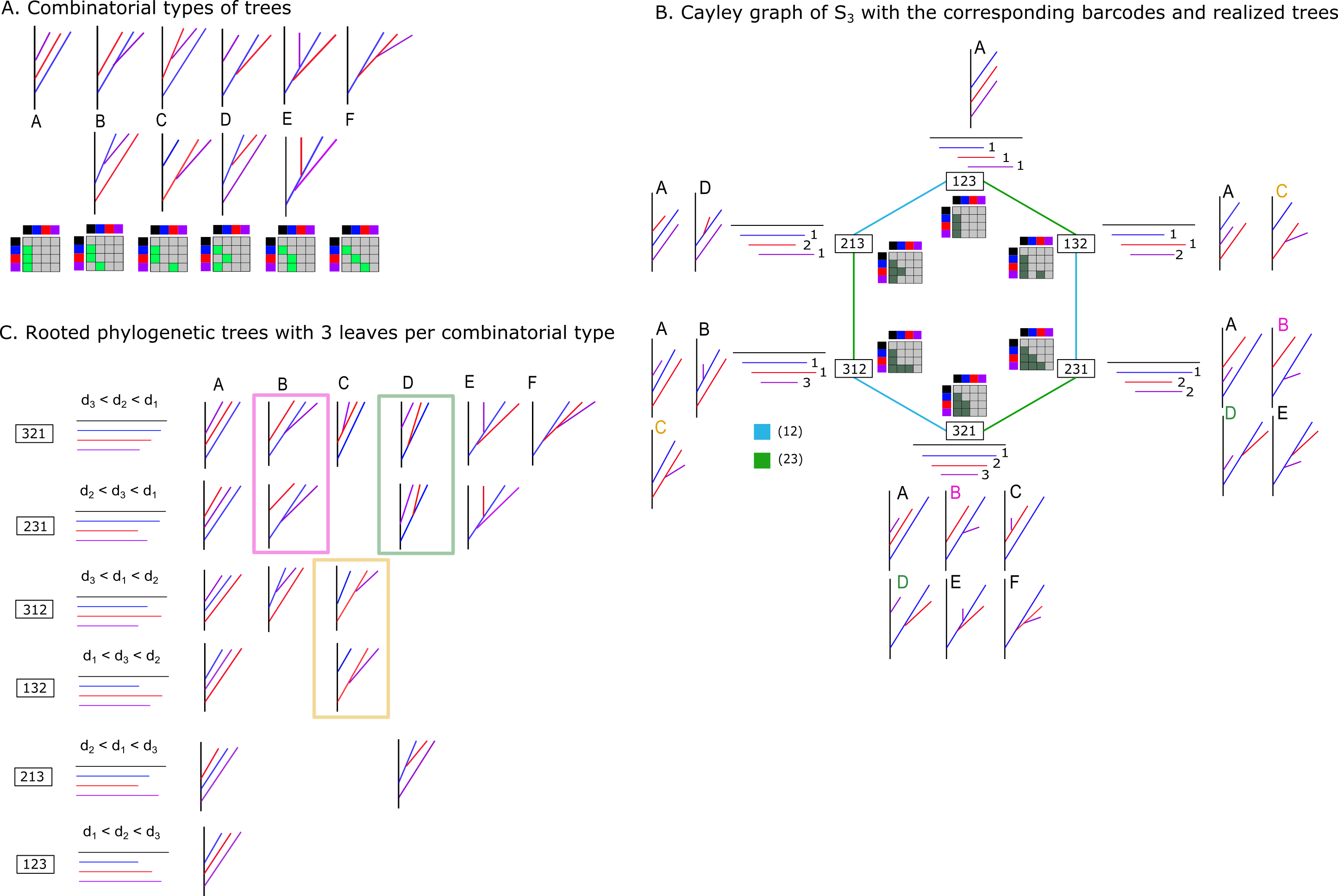}
    \caption{A. Combinatorial types of rooted trees with three leaves and the corresponding adjacency matrices. B. Cayley graph generated by the two adjacent transpositions of $S_3$ and the corresponding barcodes, together with all the combinatorial types of trees that realize a barcode. Colored letters correspond to different types of merge trees that are the same as phylogenetic trees (indistinguishable trees), illustrating the result of Section \ref{sec_MT_vs_PT}. C. Rooted phylogenetic trees with  three leaves. We represent these phylogenetic trees organised by the cominatorial types of barcodes they would have if they had death labels as well. The three pairs of trees within colored squares correspond to the indistinguishable trees defined in Section \ref{sec_MT_vs_PT}: the internal nodes are incomparable, so they can have two different death value that lead to different merge trees. In phylogenetic trees, the label order does matter: for instance, in the first column, all the trees are of the same combinatorial type A but  correspond to different phylogenetic trees. To go from the space of phylogenetic trees to the space of combinatorial trees, one forgets the labels and considers the adjacencies only, see Figure \ref{fig:elderrule}.}
    \label{cayley_averagetree}
    \vspace{1cm}
\end{figure}

\begin{proposition}
[\cite{curry2017fiber}, \cite{TRN}]\label{mtcount}
Let $B$ be a strict barcode with $n$ finite length half-open bars $\{I_j=[b_j,d_j\}_{j = 1}^n$ and one infinite bar $I_0=[b_0,\infty)$. 
The number of merge trees that realize $B$ is 
\[ 
R(B) = \prod_{j = 1}^n\mu(I_j) 
\]
where $\mu(I_j) = |\{I_k | I_j \subset I_k\}|$.
The value $\mu(I_j)$, called the \emph{index} of bar $I_j$, is the number of bars of $B$ (including the infinite bar) that contain $I_j$.
\end{proposition}

Although the proof of this theorem, by induction on the number of bars, can be found in \cite{curry2017fiber} and \cite{TRN}, we provide a brief sketch  for the sake of intuition.
Start by setting $T_0 = I_0 = [b_0, \infty)$. Since the merge tree $T$ is connected, we can recursively attach bars by death time, first to $T_0$ and then in the $j^\text{th}$ step to $T_{j}$ to get $T_{j + 1}$, according to the Elder rule. Each possible choice of attachment then gives a particular merge tree isomorphism class. See Figure \ref{snowflake} for a graphical representation of this process. 

\begin{figure}
    \centering
    \includegraphics[scale=0.15]{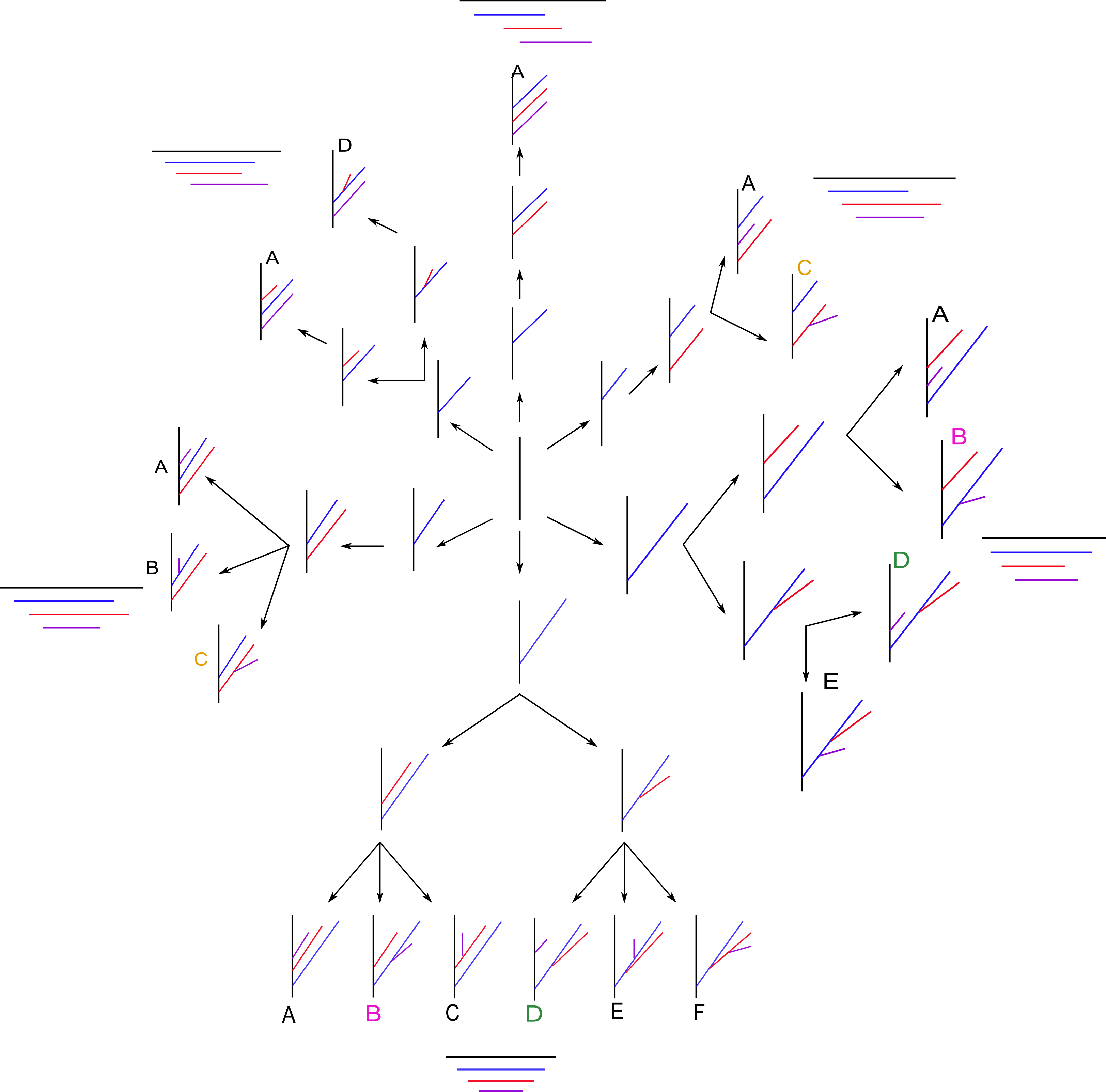}
    \caption{Recursive construction of all trees realising barcodes with three non-essential bars. At each bifurcation, the number of new branches corresponds to the index of the branch that is added. Each time we add a new branch, we multiply the number of possibilities by its index, illustrating the result of Proposition \ref{mtcount}.}
    \label{snowflake}
\end{figure}

\begin{example}
Consider the strict barcode $B =  \Big\{[0, \infty), [1, 8), [2, 7), [3, 6), [4, 5)\Big\}$. According to the formula in Proposition \ref{mtcount}, 
\[
R(B) = \prod_{j = 1}^4 \mu(I_j) = 1\cdot 2\cdot 3 \cdot4 = 4! .
\]
In general, if $B$ is a strict barcode with $n$ finite length half-open intervals such that $I_j \subset I_k$ for all $k < j$, then $R(B) = n!$.
\end{example}

\subsection{Relations to the Symmetric Group} \label{merge_tree_sym_group}

We begin by recalling the map from the set of strict barcodes with $n$ nonessential bars to the symmetric group on $n$ letters, which was introduced in \cite{TRN}. 

%The notation will be slightly different, as their main purpose was to use this bijection as a tool for classifying trees, and here we focus on the combinatorial and probabilistic aspects. 
\begin{remark}[Different Notations for Permutations]
There are several notational conventions for elements of the symmetric group.
When we use square brackets or boxes, e.g.,~the notation $[132]$, then we are listing the images of the ordered set $\{1,\ldots, n\}$ under the map $\sigma$, e.g., for $\sigma=[132]$, one can read off that $\sigma(1)=1$, $\sigma(2)=3$ and $\sigma(3)=2$.
We also use cycle notation, which describes the permutation in terms of its orbits and uses parentheses; fixed points are omitted in this notation. For our example, $\sigma=[132]$ can also be written as the elementary transposition $(23)$.
See Figure \ref{fig:symmetric-group-PDs}.
\end{remark}

\begin{definition}
Let $B=\{[b_i,d_i)\}_{i=1}^n\cup [b_0,\infty)$ be a strict barcode such that $b_1<...<b_n$. 
%If we fix the birth times, then there are $n!$ possible configurations of the bars, one for each ordering of death times $d_{i_1} <  \cdots < d_{i_n}$. 
The \emph{permutation type} $\sigma$ of the barcode $B$ is the automorphism $\sigma$ of $\{1,\ldots,n\}$ that maps birth order to death order.
In other words, if we re-index the death times using the natural order on $\R$ so that
$d_{i_1} <  \cdots < d_{i_n}$, the permutation $\sigma$ is $[i_1 i_2...i_n]$. In terms of the Elder rule, this \emph{associated permutation} comes from tracking which birth is paired with which death.
% The map $\sigma$ captures the information of the $j$-th birth being paired with the $k$-th death. We say that this permutation $\sigma$ is \emph{associated} to the barcode $B$. 
% This establishes a bijection between certain generic configurations of the bars in a barcode and elements of $S_n$.
\end{definition}

Notice that the essential bar $[b_0,\infty)$ does not play a role in the permutation type, as it always contains all the other bars in a strict barcode.

\begin{figure}[h]
    \centering
    \includegraphics[width =.5 \textwidth]{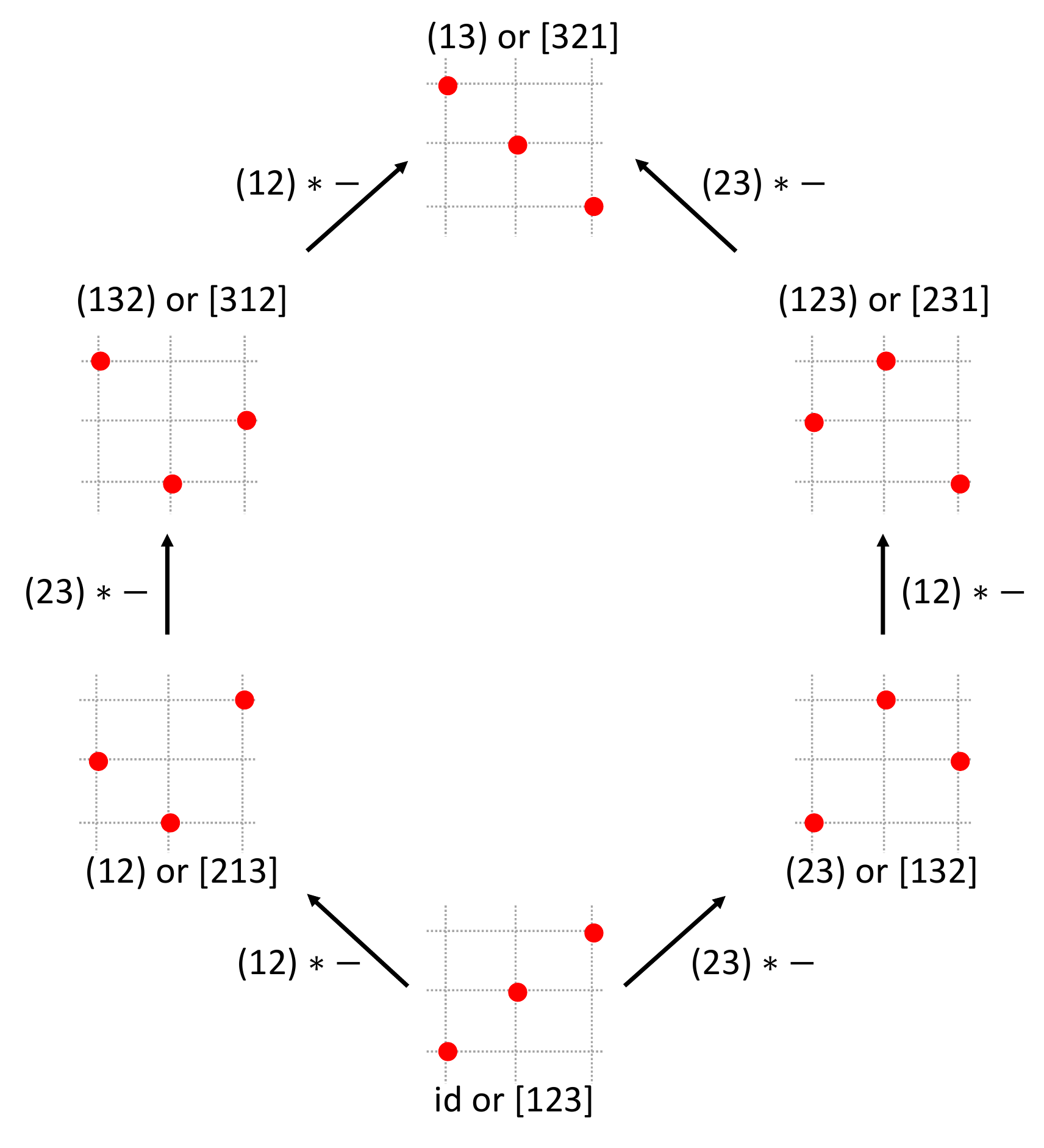}
    \caption{Combinatorial equivalence classes of persistence diagrams with three non-essential points. The associated permutation $\sigma$ is written next to each diagram in both forms of notation: the image notation is in square brackets, i.e,.~$[\sigma(1)\sigma(2)\sigma(3)]$, and the cycle notation  in parentheses. The arrows point in the direction of increasing left Bruhat order and exhibit $S_3$ as a poset. Notice that the permutation acts by switching death order.}
    \label{fig:symmetric-group-PDs}
\end{figure}

%Note that the ordering of the births is not necessary to define $\sigma$, it is simply more convenient to assume that the births are ordered.
The association of a permutation to each barcode defines an equivalence relation on the set of strict barcodes.

\begin{definition}\label{defn:combo-equiv-barcodes}
Let $B$ and $B'$ be two strict barcodes, each with $n$ non-essential bars, denoted $\{ [b_i,d_i) \}_{i = 1}^n$ and $\{ [b'_i,d'_i) \}_{i = 1}^n$, respectively. 
We say $B$ and $B'$ are \textit{combinatorially equivalent} if they have the same associated permutation.
\end{definition}

%Note that because the half infinite bar $[b_0,d_0=\infty)$ always contains all the others and has no finite death time, we do not take it into account and consider barcodes with $n+1$ bars for the bijection with the symmetric group $S_n$.

We can now express the relation between barcodes and the symmetric group more concisely as follows.
Let $\mathcal{B}_n$ denote the collection of strict barcodes with $n$ finite length half-open (non-essential) bars. 
The map that associates to every strict barcode its permutation type defines a bijection between combinatorial equivalence classes of strict barcodes and elements of the symmetric group, i.e.,
\[ 
\mathcal{B}_n /\sim  \qquad \longleftrightarrow \qquad  S_n.
\]

% This bijection is established by mapping each strict barcode with its associated permutation:
% \[ 
% B = \{[b_i, d_i)\}_{i = 0}^n \mapsto \{\sigma_B: i \mapsto \#\{j \mid d_j < d_i\}\}_{i = 1}^n .
% \] 
% Going forward, we will often not distinguish between a barcode $B$ and its associated permutation $\sigma$.

\begin{example} 
The space $\mathcal{B}_3 /\sim $ and the corresponding elements of $S_3$ of the bijection given above are displayed in Figure \ref{cayley_averagetree}C.
\end{example} 

\begin{remark}
As was done in Remark \ref{rmk_comb_MT} for combinatorial merge trees, one can identify the combinatorial equivalence classes of barcodes with elements of the symmetric group. What will be called a \emph{combinatorial barcode} in this paper is just the corresponding permutation in $S_n$.
\end{remark}

We conclude this section by clarifying the relationship between the two notions of combinatorial equivalence that are pertinent to the tree realization problem.

\begin{lemma}\label{lem:combo-trees-and-BCs}
If $T$ and $T'$ are combinatorially equivalent merge trees, then their corresponding barcodes $B$ and $B'$ are combinatorially equivalent as well.
\end{lemma}

\begin{proof}
 Since tree isomorphisms as defined in Definition \ref{fig:combo-equiv-MTs} preserve both birth and death orders, we need to check only that if the Elder rule pairs the $i$-th birth node with the $j$-th death node in $T$, then the same holds for $T'$. This is obvious, however, because the unique sequence of edges connecting a pair  of nodes in $T$ must be sent to the same sequence of edges connecting these nodes in $T'$, since $\varphi$ is a graph isomorphism and therefore preserves adjacencyy relations.
\end{proof}

Figure \ref{fig_comb_trees_barcodes} illustrates the relationship between merge trees and their combinatorial equivalence classes and barcodes and their combinatorial equivalence classes, corresponding to permutations.

\begin{figure}
    \centering
    \includegraphics[width =.8 \textwidth]{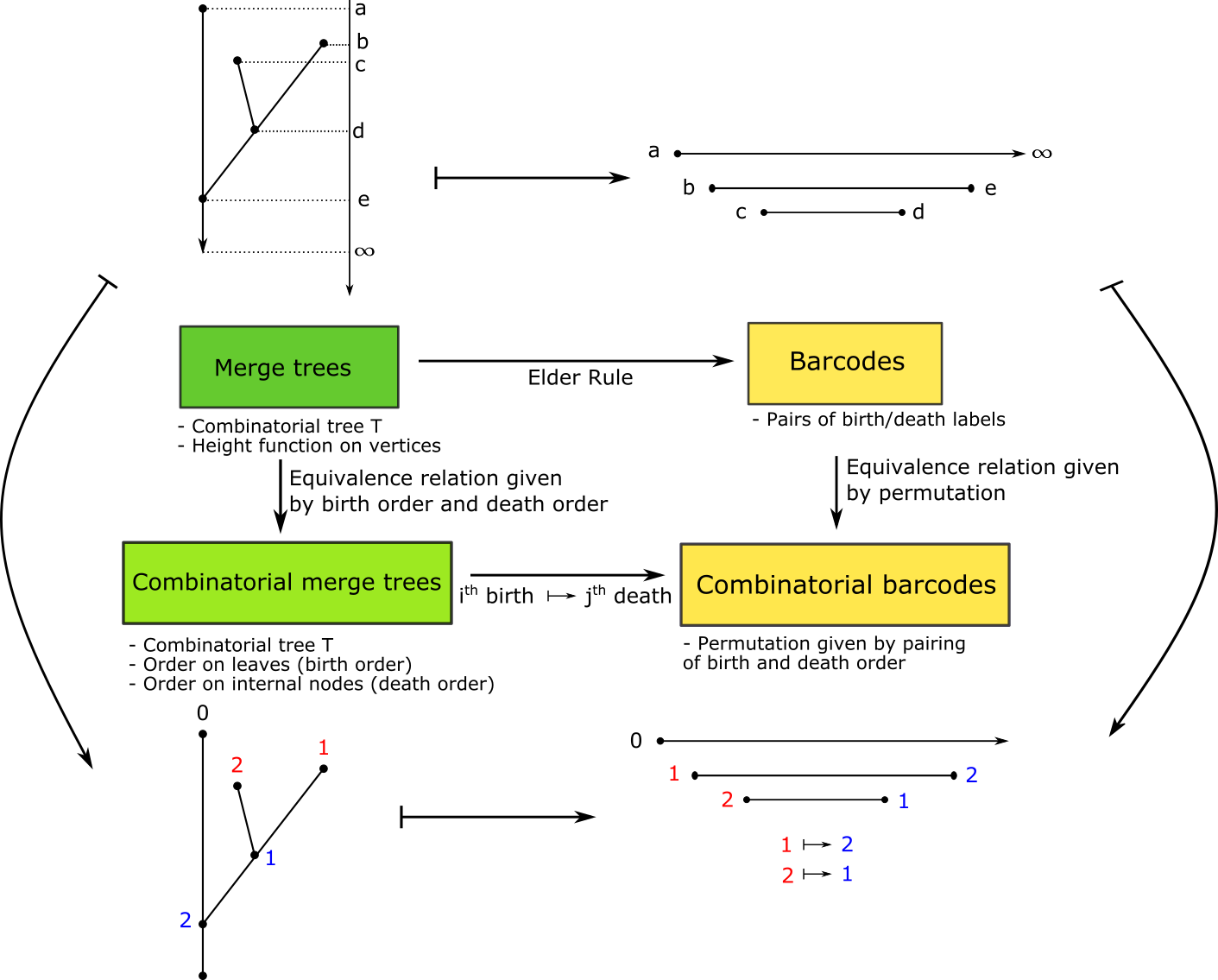}
    \caption{The relationships between merge trees, combinatorial equivalence classes of merge trees, barcodes and combinatorial barcodes. Birth labels are indicated in red, and death labels in blue. The largest bar (corresponding to the essential class) is not taken into account in the combinatorial setting since it is there for every tree/barcode. Therefore we label it by $0$.}
    \label{fig_comb_trees_barcodes}
\end{figure}

\section{Combinatorial and Algebraic Perspectives on the Realization Number}\label{sec:combo-alg-TRN}

Now that we have reviewed the basic notions of trees, merge trees, their barcodes, and prior results on the inverse problem detailed in \cite{curry2017fiber} and \cite{TRN}, we are in a position to extend those results.
The first observation of this section is that the tree realization number (TRN) of a barcode is simply the product of the entries of the left inversion vector for the permutation associated to a barcode. 
This is somewhat surprising, as the left inversion vector is a classical object of study, but typically authors study the sum of its entries rather than the product.
This observation also allows us to characterize those barcodes that have a larger tree realization number in the language of geometric group theory: permutations that have longer word length in the left Bruhat order have higher TRN.
Based on a convexity result for combinatorial equivalence classes of barcodes, we also provide a closed form expression for the sum of TRNs across all elements of the symmetric group, which is equal to the number of maximal chains in the lattice of partitions.
This result is of use in the next section, when we consider probability distributions on the space of barcodes and calculate the expected tree realization number for the uniform distribution on the symmetric group.

\subsection{The Realization Number and the Left Inversion Vector} \label{sec_real_number}

%It turns out that the realization number of a barcode has a natural combinatorial interpretation in the language of permutations.
Careful inspection of the formula for the tree realization number in Proposition \ref{mtcount} reveals that the index of a bar $[b_i,d_i)$ in a barcode $B$ is given by the number of bars born before $b_i$ and that die after $d_i$. 
Thinking in terms of the permutation associated to a barcode, this index counts the number of ``upsets'' of birth-mapping-to-death order. 
%Consequently, it seems natural to look to \emph{inversions}, which give similar information about a permutation. 
More precisely, for a permutation $\sigma$ of $\{1,\ldots,n\}$ if $i < j$ and $ \sigma(i) > \sigma(j)$, then either the pair of places $(i,j)$ or the pair of elements $(\sigma(i),\sigma(j))$ is called an \textit{inversion} of $\sigma$---the usual order $i < j$ has been ``upset'' or inverted here.
We now modify the usual notion of an inversion vector so that it is defined for strict barcodes and makes our theorem statements as tidy as possible.

\begin{definition}\label{defn:left-inversion-vector}
Let $B=[b_0,\infty) \cup \{[b_i,d_i)\}_{i=1}^n$ be a strict barcode with $b_i < b_j$ for $i< j$.
The \textit{left inversion vector} of $B$ is the $n$-vector $l(B)$ whose $i$-th coordinate is 
\[
l_i(B):=\#|\{j \leq i \mid d_j \geq d_i\}|.
\]
We note that for this formula the index $j=0$ is used for computation although it is not given a position in the $n$-vector $l(B)$, since the vector would have length $n+1$.
When we calculate the left inversion vector of a permutation $\sigma$ associated to a barcode, we use the slightly modified definition
\[
l_i(\sigma):=\#|\{j \leq i \mid \sigma(j) \geq \sigma(i)\}|
\]
in order to make sure that $l(\sigma)=l(B)$.
\end{definition}

%\begin{remark}
%We note that the above definition differs from the usual definition of the left inversion vector associated to a permutation $\sigma$ of $\{1,\ldots,n\}$, which is 
%\[
%\ell_i(\sigma):= \#|\{j < i \mid \sigma(j) > \sigma(i)\}|,\]
%because every entry has the value $1$ added to it.
%This added value of $1$ comes from the fact that that the Elder rule dictates that the first thing born is the essential class and never dies.
%\end{remark}

 \begin{figure}
     \centering
     \includegraphics[width=\textwidth]{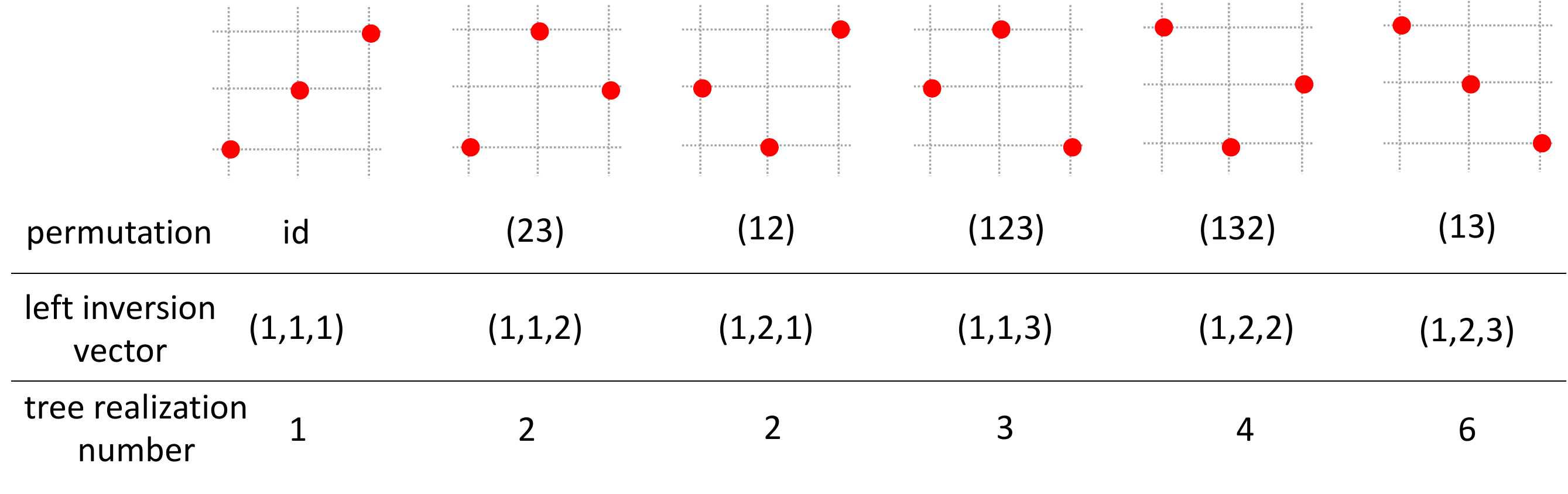}
     \caption{Persistence diagrams associated to the six elements of $S_3$, along with their inversion vector and tree realization number.}
     \label{fig:PD-left-inversion-vector}
 \end{figure}

\begin{example}
One can easily compute the left inversion vector of the following barcode with one essential class and four non-essential classes: 
\[
B=\{[0,\infty), [1, 7), [2, 6), [3, 5), [4, 8) \} \quad \Rightarrow \quad l(B) = (1,2,3,1).
\]
The permutation associated to this barcode is $\sigma = (3 2 1 4)$ because the first non-essential feature dies third, the second feature dies second, the third feature dies first and the fourth feature dies fourth.
Clearly, $l(\sigma) = (1, 2, 3, 1)$ as well.
\end{example}

\begin{example}
For the left inversion vectors associated to the six elements of $S_3$, along with their tree realization numbers, see Figure~\ref{fig:PD-left-inversion-vector}.
\end{example}

To define coordinates on the space of left inversion vectors, we use the  the totally ordered sets
\[
    [k]:=\{1 < 2 < \cdots < k\}
\]
for $k$ a positive natural number.
It is easy to see that the left inversion vector construction establishes a bijective correspondence between $S_n$ and the Cartesian product of sets of the above form, i.e., there is a bijection
\[
    l: S_n \longrightarrow [1]\times[2]\times \cdots \times[n-1]\times[n] \qquad \text{where} \qquad \sigma \mapsto l(\sigma).
\]
The next lemma, which is crucial for the rest of the paper, follows immediately from this observation. It was first established in \cite{TRN}, though not formulated explicitly in terms of the left inversion vector.

\begin{lemma}\label{lem:TRN-left-inversion}
If $B$ is a strict barcode with one essential bar $[b_0,\infty)$ and $n$ non-essential bars $\{[b_i,d_i)\}_{i=1}^n$, then
\[ 
R(B) = \prod_{i = 1}^n l_i(\sigma(B)).
\]
\end{lemma}

An immediate consequence of this lemma is that if $B$ and $B'$ are combinatorially equivalent barcodes, in the sense of Definition \ref{defn:combo-equiv-barcodes}, then their realization numbers are the same. It follows that the tree realization number induces a function on the symmetric group, i.e.,
\[
R: S_n \to \N : \sigma \mapsto \prod_{i = 1}^n l_i(\sigma).
\]
Before analyzing this function on the symmetric group, we identify some interesting properties of the set of barcodes under the combinatorial equivalence relation, to prepare our exploration of the combinatorics of the TRN in earnest in subsequent sections.
%%%%%%%%%%%%%%%%%%%%%%%%%%%%%%%%%%%%%%%%%%%%%%%%%%%%%%%%%%%%%%

\subsection{Convexity of Combinatorial Equivalence Classes} \label{sec_continuous_path}

In this section we prove that combinatorial equivalence classes are convex in a certain sense: if two strict barcodes $B$ and $B'$ are of the same combinatorial type, then they can be connected by a ``line segment'' of barcodes\footnote{A continuous path of barcodes is sometimes called a \emph{vineyard}. This terminology arises more commonly when barcodes are represented using persistence diagrams, as this path traces out a configuration of points in the plane, with points appearing and disappearing out of the diagonal.} all of the same permutation type. 
%Moreover this path has the structure of a straight line homotopy, which implies that the equivalence classes of barcodes are connected and convex.

We prove first that the set $\mathcal B_n$ admits the algebraic structure necessary to formulate a convexity result.  

\begin{lemma}\label{lemma_Bn_module} 
\begin{enumerate}
    \item For all $\lambda \in \R_{>0}$ and $B=\{[b_0,\infty)\}\cup \{[b_i,d_i)\}_{i=1}^n\in \mathcal B_n$, the set $$\lambda B:=\{[\lambda b_0,\infty)\}\cup \{[\lambda b_i,\lambda d_i)\}_{i=1}^n$$ is also a strict barcode.
    \item For all $B=\{[b_0,\infty)\}\cup \{[b_i,d_i)\}_{i=1}^n,B'=\{[b'_0,\infty)\}\cup \{[b'_i,d'_i)\}_{i=1}^n \in \mathcal B_n$, the set $$B+B':=\{[b_0+b_0',\infty)\}\cup \{[b_i+b_i',d_i+d_i')\}_{i=1}^n$$ is also a barcode with distinct birth times, which is strict if $B$ and $B'$ have the same permutation type.
\end{enumerate}
\end{lemma}

\begin{proof}
The proof of (1) is trivial, since $\lambda$ is assumed to be positive, whence multiplication by $\lambda$ preserves the order of real numbers.

The only subtlety in the proof of (2) concerns distinct death times.  If the permutation types of $B$ and $B'$ are different, it could happen that $d_i<d_j$ and $d'_i>d_j'$, but $d_i+d'_i=d_j+d'_j$, so that $B+B'$ would not be strict.  If they have the same permutation type, then this cannot happen. 
\end{proof}

\begin{lemma}\label{lemma_cont_path} For every $n$ and every $\sigma \in S_n$, the set of strict barcodes of permutation type $\sigma$ is convex, i.e.,  for $B$ and $B'$ of permutation type $\sigma$, the interval
$$[B,B']:=\{ tB+(1-t)B'\mid t\in [0,1]\}$$
is contained in the set of barcodes of permutation type $\sigma$.
\end{lemma}
 
\begin{proof}
Given the previous lemma, it remains only to prove that the permutation type of $tB+(1-t)B'$ is $\sigma$, which follows immediately from the observation that 
$$d_i<d_j \text { and } d_i'<d_j' \Longrightarrow d_i+d'_i < d_j + d'_j.$$
\end{proof}

\begin{remark}\label{remark_path_barcodes}
We can also formulate the lemma above as saying that there is a ``straight-line path" from $B$ to $B'$,
$$ \overline{BB'}:[0,1] \to \mathcal B_n: t \mapsto B^t:= tB+(1-t)B'.$$
It is not hard to show that this function is indeed continuous with respect to both the bottleneck metric and the Wasserstein metric on $\mathcal B_n$, but we choose not to do so here, to avoid introducing further definitions outside of the focus of this paper.

It is interesting also to consider the path $\overline{BB'}$ when the barcodes $B$ and $B'$ are not of the same permutation type. As mentioned in the proof of Lemma \ref{lemma_Bn_module}, not every point of $\overline{BB'}$ is necessarily a strict barcode in this case, which allows the path to move from one permutation type to another.
One can show that the smallest number of different classes that the path goes through is the length of the shortest path between the two corresponding permutations of $B$ and $B'$ on the Cayley graph defined using the generating set of elementary (neighboring) transpositions $\tau_i=(i,i+1)$. 
This value is related to the Bruhat order, which we introduce in the next section.
A fuller description would involve describing the space of barcodes in terms of a family of convex sets that fiber over the permutohedron. 
We leave this for future work.
\end{remark}
 
 \begin{example}
 Figure \ref{fig_path_diagrams} shows an example of the path described in the proof above, using the representation of barcodes as persistence diagrams. The path consists of the straight lines between the matched points of the diagrams. Note that the dotted lines indicating the births and deaths never cross for the same birth and death order, respectively, because the barcodes stay in the same permutation class at each step of the path. 
It \emph{is} possible for $b_1$ to be greater than $b_2'$, for example, but the relative order of births and deaths does not change.

 \begin{figure}
     \centering
     \includegraphics[scale=0.8]{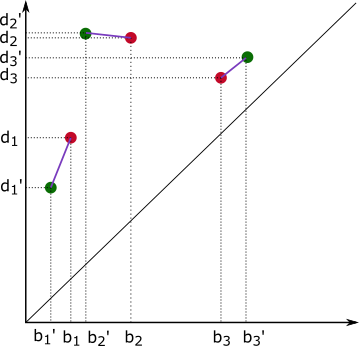}
     \caption{Continuous path between two barcodes in the same combinatorial class. We show the persistence diagrams of each barcode. The first one $B$ is indicated by red dots and the second one by green dots, and the path $\mathcal{B}^t$ is in purple. }
     \label{fig_path_diagrams}
 \end{figure}
 \end{example}

Lemma \ref{lemma_cont_path}  allows us to fix a standardized representative of each combinatorial barcode type, making  the connection to the symmetric group explicit. 

\begin{definition}\label{defn:standard-form-BC}
A barode $B$ is in \emph{standard form} if there is a permutation $\sigma$ of the set $\{1,\ldots,n\}$ so that
\[
B= \{[i,\sigma(i)+n) \}_{i = 1,...n} \cup \{[0,\infty)\}.
\]
It is clear that $B$ is strict and has permutation type $\sigma$.
We sometimes write $B(\sigma)$ for the standard barcode associated to $\sigma$.
\end{definition}

Lemma \ref{lemma_cont_path} implies that any strict barcode $B$ of permutation type $\sigma$  can be connected via a straight-line path  to the barcode $B(\sigma)$.

%%%%%%%%%%%%%%%%%%%%%%%%%%%%%%%%%%%%%%%%%%%%%%%%%%%%%%%%%%%%%%    
\subsection{Tree Realization Number Preserves Bruhat Order}\label{sec:Bruhat}

It is interesting to study both the tree realization number from a combinatorial point of view via the symmetric group and the symmetric group from a ``barcode'' point of view via the realization number. 
To our knowledge, the product of the components of the left inversion vector is not a very commonly used statistic on symmetric groups, so we take this opportunity to study some of its properties. 

Observe first that two adjacent permutations in the Cayley graph (i.e., two permutations that differ by left multiplication by one elementary transposition $\tau_i=(i,i+1)$) never have the same realization number.  This follows easily from the definition. 
As a consequence, the realization number is locally injective, although it is not globally injective, since barcodes of type $(12)$ and type $(23)$ have the same TRN. 
%close barcodes in terms of permutation distance (path-distance on the Cayley graph) do not have similar realization number. However, it is not necessarily a bad consequence. It means that the realization number can be used to differentiate barcodes that can be close to each other, as was done in \cite{TRN} to study the transposition stability. 
In this section we extend this local injectivity observation, proving that the TRN defines an order-preserving map from the symmetric group to the natural numbers, when the symmetric group is equipped with the appropriate Bruhat order.

Recall that the symmetric group is generated by elementary transpositions $\tau_i:= (i,i+1)$. 
This implies that any element of $S_n$ can be represented using a word made using the alphabet $\mathcal{A} = \{(i,i+1)\}_{i=1}^{n-1}$, although that representation need not be unique. A word representing a certain permutation is \emph{reduced} if it is of minimal length.  The \emph{length} of a permutation is the minimal length of a word representing the permutation.

\begin{definition}[Left Bruhat Order]\label{defn:left-Bruhat}
The \emph{left Bruhat order} is a partial order on $S_n$, specified as follows.
If $\sigma, \sigma' \in S_n$, then $\sigma < \sigma'$ if the length of $\sigma$ is less than that of $\sigma'$, and there exist $\tau_{i_1}, ...,\tau_{i_k}\in \mathcal A$ such that $\sigma' = \tau_{i_1}\cdots \tau_{i_k}\sigma$.
%The \emph{strong Bruhat order} is defined as the reflexive and transitive closure of the relation $\sigma <_S (i, j) \sigma$ for pairs $i <j$.
\end{definition}

%\textbf{ACHTUNG! THERE ARE TWO NOTIONS OF BRUHAT ORDER HERE, BUT BOTH SEEM TO BE PRESERVED BY TRN!}

\begin{example}
% In $S_3$ we note that $(123) > (13)$ under the left Bruhat order because $(13) = (23)(123)$, where we use cycle notation and where composition is read from right to left.
In $S_3$ we note that $(123) > (23)$ under the left Bruhat order because $(123) = (12)(23)$, where we use cycle notation and where composition is read from right to left.
% This should be contrasted with the strong Bruhat order which has the extra relation $(123) >_S (23)$ because $(123)$ can also be written as $(13)(23)$.
% This extra covering relation comes from using the non-adjacent transposition $(13)$.
In the left Bruhat order $(123)$ and $(12)$ are not comparable; see Figure \ref{fig:symmetric-group-PDs}.
\end{example}

The next lemma shows that the realization number increases with increasing  left Bruhat order. We remark that this lemma can be viewed as a consequence of a classical result, which is mentioned in~\cite{Bruhat-Order-Sn}: if $\sigma < \sigma'$, then the number of inversions in $\sigma'$ is greater than the number of inversions in $\sigma$.

\begin{lemma}
If $\sigma, \sigma' \in S_n$ are such that $\sigma < \sigma'$ in the left Bruhat order, then $R(\sigma) < R(\sigma')$.
\end{lemma}

\begin{proof}
Since $\sigma < \sigma'$, there exist $\tau_{i_1}, ...,\tau_{i_k}\in \mathcal A$ such that $\sigma' = \tau_{i_1}\cdots \tau_{i_k}\sigma$. If $k=1$, so that $\sigma$ and $\sigma'$ are adjacent on the Cayley graph, i.e., $\sigma'=\tau_i \sigma$ for some $i$. 
By assumption, the length of $\sigma'$ is greater than that of $\sigma$. 
%Therefore, $\sigma'(i+1) > \sigma'(i)$ while $\sigma(i+1)< \sigma(i).$ 

Translating Proposition 3.5 in \cite{TRN} into the language of permutations, we deduce that 
\[ 
R(\sigma') = \frac{R(\sigma) (l_{i+1}(\sigma) + 1)}{l_{i+1}(\sigma)} > R(\sigma).
\]
The result now follows by induction on the number of transpositions $\tau_i$. 
\end{proof}

\begin{example}
One can see the Cayley graph of $S_4$ in Figure \ref{fig_cayley_s4}. Notice that two permutations $\sigma, \sigma'$ satisfy $\sigma < \sigma'$ in the Bruhat order if and only if the shortest path from $\sigma'$ to the identity contains the shortest path from $\sigma$ to the identity. The realization number increases along such paths.

\begin{figure}
    \centering
    \includegraphics[scale=0.6]{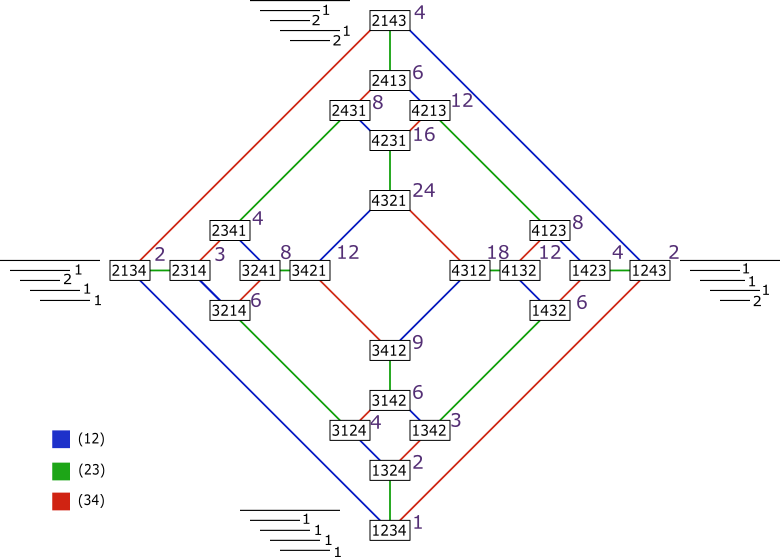}
    \caption{The Cayley graph of the symmetric group $S_4$ and some of the corresponding barcodes. We indicate the tree realization number in purple next to each vertex. }
    \label{fig_cayley_s4}
\end{figure}
\end{example}

\begin{remark}
It is interesting to consider the TRN as a discrete Morse function \cite{FORMAN199890} on the order complex of $S_n$. We note that the TRN has a unique max and min on $S_n$, which appear to be the only critical points, recovering the known result, e.g.~\cite{Bruhat-Order-Sn}, that the order complex of $S_n$ is homotopy equivalent to a sphere.
\end{remark}

\subsection{The Sum of Realization Numbers and Chains in the Lattice of Partitions}\label{sec:sum}

Given that the tree realization number on the set of strict barcodes induces a function $R:S_n \to \N$, it is natural to study the  sum:
\[ 
\sum_{\sigma \in S_n}  R(\sigma). 
\]
As we show in this section, this sum is equal to the number of combinatorial classes of merge trees (Definition \ref{defn:combo-merge-tree}) and provides another quantitative characterization of the difference between merge trees and phylogenetic trees, which is explored further in the next section.

The sum of TRNs also connects this work with a classical object of study in algebraic combinatorics: each combinatorial equivalence class of merge trees corresponds to a maximal chain in the lattice of partitions, ordered by refinement.
For topologists this should make intuitive sense: as two connected components merge this coarsens the partition of a sublevel set into connected components. Enumerating these components leads naturally to the study of the partitions of the set of $\{0,1,\ldots, n\}$.

We start now by showing that this sum counts combinatorial equivalence classes of merge trees, but first prove a preparatory lemma.

%Using the observations in Example \ref{BarstoTreesExample}, we are able to relate the sum of realization numbers to the maximal chains in the lattice of partitions of $\{0, 1, 2 \}$ ordered by refinement:

\begin{lemma}\label{lem:line-of-MTs}
If $(T,h)$ and $(T',h')$ are combinatorially equivalent merge trees with associated barcodes $B$ and $B'$, then the straight-line path $\overline{BB'}$ from $B$ to $B'$ lifts to a continuous path (with respect to the interleaving distance) connecting $T$ and $T'$.
\end{lemma}
\begin{proof}
Lemma \ref{lem:combo-trees-and-BCs} guarantees that the barcodes $B$ and $B'$ associated to $T$ and $T'$ have the same permutation type, so that the straight-line path $\overline{BB'}$ of Remark \ref{remark_path_barcodes} does indeed exist, and every point on the path is a barcode of that permutation type by Lemma \ref{lemma_cont_path}.
We now apply the Elder Rule to construct  a one-parameter family of merge trees 
$$[0,1]\to \mathcal{MT}_n: t\mapsto (T^t,h^t)$$ 
that lifts the path $\overline{BB'}$.

Since $(T,h)$ and $(T',h')$ are combinatorially equivalent, the trees $T$ and $T'$ are isomorphic as graphs.  Without loss of generality, we can suppose that $T=T'$. 

To define our one-parameter family of  merge trees, we set $T^t=T$ for all $t\in [0,1]$ and  specify the height function $h^t:V(T)\to \R$ as follows.
We have no choice but to set $h^t(r)=\infty$, where $r$ is the root, so it remains only to define $h^t$ on the non-root nodes.

If $v_i$ is the $i$-th leaf node by birth order in $T$, and therefore corresponds to the $i$-th bar of $B^t$, then the $h^t(v_i)$ is chosen to be the birth time of this bar, i.e., 
\[
    h^t(v_i)=b_i(1-t) + b'_i t.
\]
Similarly, if $w_i$ is the internal node corresponding to the $i$-th bar in $B^t$, then $h^t(w_i)$ is chosen to be the death time of this bar, i.e.,
$$h^t(w_i)=d_i(1-t) + d'_i t.$$
By construction, the barcode associated to $(T,h^t)$ is clearly $B^t$.

It was shown in \cite{interleaving} (Theorem 2.2) that the interleaving distance between two merge trees in bounded by the maximal difference between the two height functions. Since $T^{t_1}=T^{t_2}$ for all $t_i \in [0,1]$ and the height functions $h^t$ change continuously with respect to the $l_\infty$ norm, it follows that the path defined by $t \mapsto (T^t,h^t)$ in the space of trees is continuous.
\end{proof}

%Lemma \ref{lem:line-of-MTs} proves that any merge tree whose barcode is of permutation type $\sigma$ can be connected with a line of merge trees to one that is in a particular form.

\begin{definition}\label{defn:standard-form-MT}
We say that a generic merge tree $(T,h)$ with $n+1$ leaves is in \emph{standard form} if its height function $h$ maps its leaf nodes onto $\{0,1,\ldots,n\}$ and its internal non-root nodes onto $\{n+1,\ldots, 2n\}$. 
\end{definition}

It is clear that a merge tree in standard form has a barcode in standard form (Definition \ref{defn:standard-form-BC}).

\begin{lemma}\label{lem:count-combo-classes-MTs}
For all $\sigma \in S_n$, the tree realization number $R(\sigma)$ is equal to the number of combinatorial equivalence classes of merge tree whose barcode has permutation type $\sigma$.
\end{lemma}

It follows immediately from this lemma that
\[
\sum_{\sigma \in S_n}  R(\sigma) =\# \{\text{combinatorial classes of merge trees}\},
\]
since barcode permutation type is also an invariant of the combinatorial equivalence type of the merge tree.

\begin{proof}
By Lemma \ref{lem:line-of-MTs} there is a path in $\mathcal{MT}_n$ from any merge tree whose barcode is of permutation type $\sigma$ to one that is in standard form (Definition \ref{defn:standard-form-MT}).

The tree realization number $R(\sigma)$ counts the number of merge trees in standard form with the standard form barcode $B(\sigma)$; see Definition \ref{defn:standard-form-BC}.
%Both \cite{curry2017fiber} and \cite{TRN} prove that each merge tree with barcode $B(\sigma)$ is non-isomorphic because each choice of realization produces a merge tree with at least one edge of different length, and edge length is an isomorphism invariant for merge trees.
If two different merge trees $(T,h)$ and $(T',h')$ are both in standard form with the same barcode $B(\sigma)$, then they cannot be combinatorially equivalent. The inductive construction that created $T$ and $T'$ must have differed in a choice for some $i\in\{1,\ldots, n\}$ of where to attach a branch with leaf node at height $i$: to a branch with leaf node at height $j$ or height $j'$, with $0\leq j\neq j' < i$. An isomorphism of merge trees from $(T,h)$ to $(T',h')$ would have to exchange the order of of the leaf nodes at heights $j$ and $j'$, which is prohibited by the definition of combinatorial equivalence of merge trees (Definition \ref{defn:combo-merge-tree}).
\end{proof}

Since every merge tree is combinatorially equivalent to one in standard form, where leaf nodes are at heights $\{0,1,\ldots, n\}$, we can use this positioning to relate merge trees with maximal chains in the lattice of partitions of $n$.
We review briefly the necessary definitions.

\begin{definition}\label{defn:lattice-of-partitions}
A \emph{partition} of the set $\mathbf{n}:=\{0,1,\ldots,n\}$
is a collection of pairwise disjoint subsets $\mathcal{U}=\{U_1,\ldots, U_k\}$ of $\mathbf{n}$ whose union is $\mathbf{n}$.
A partition $\mathcal{U}$ \emph{refines} a partition $\mathcal{U}'$, written $\mathcal{U}\preceq \mathcal{U}'$, if every subset of $\mathcal{U}'$ is equal to a union of elements of $\mathcal{U}$. 
Said differently, $\mathcal{U}\preceq \mathcal{U}'$ if for each $U_i\in \mathcal{U}$ there exists $U_j'\in\mathcal{U}'$ such that  $U_i\subseteq U_j'$.
We denote the set of partitions of $\mathbf{n}$ by $\mathcal{P}_n$.
The refinement relation endows the set $\mathcal{P}_n$ with a partial order, which also happens to be a lattice.
A \emph{chain} in the lattice of partitions is a sequence of comparable partitions 
\[
\mathcal{U}_1\preceq \cdots \preceq \mathcal{U}_{\ell}.
\]
Such a chain is \emph{maximal} if it is not a subsequence of any longer chain.
\end{definition}

For the sake of notation, we can always write a partition of $\mathbf{n}$ as an ordered list where each subset is separated by a vertical line.
The finest possible partition---and hence the bottom element of the $\mathcal{P}_n$---is denoted
\[
    \{0|1|2|\cdots|n\}.
\]
The top element of $\mathcal{P}_n$ is the set $\{\mathbf{n}\}$.

\begin{theorem}\label{MaxChain}
Combinatorial equivalence classes of merge trees with $n+1$ leaf nodes are in bijective correspondence with maximal chains in the lattice of partitions $\mathcal{P}_n$.
As a consequence, the sum of realization numbers is given by the following closed form formula:
\[  \sum_{\sigma \in S_n} R(\sigma) = \sum_{\sigma \in S_n}\prod_{i=1}^n l_i(\sigma) = \dfrac{(n + 1)!n!}{2^{n}}.  
\]
\end{theorem}

\begin{proof} Given a merge tree $(T,h)$ in standard form with $n+1$ leaves, we explain first how to construct an associated maximal chain in the lattice of partitions, $\mathcal{P}_n$.  We then show that  every maximal chain is associated to some merge tree and that non-equivalent trees gives rise  to distinct maximal chains.   

Since $(T,h)$ is in standard form, all of the merge events (bifurcations) happen after (are at greater height than) all the birth events.
It follows that the sublevel set of $h:V(T)\to\R$ at any value in the interval $(n,n+1)\subset \R$ consists of $n+1$ components, corresponding to the finest partition $\mathcal{S}(T)_1:=\{0|1|2|\cdots|n\}$.

As we cross height $n+1$, the definition of the standard form implies that a merge event of two components, born at heights $i$ and $j$, occurs.
This merge event has the effect of coarsening the partition $\mathcal{S}(T)_1$, placing the two elements $i$ and $j$ into a single set of the partition. This defines the next, coarser partition $\mathcal{S}(T)_2$.

In general the $i$-th partition associated to the tree $T$ is the partition of the leaf nodes into connected components at height $n+i$.
At height $2n$ the sublevel set of the tree is connected, which corresponds to the top element in $\mathcal{P}_n$.

Each standard form merge tree thus gives rise to a chain of $2n$ elements in $\mathcal{P}_n$, which is obviously maximal.  Moreover, from any maximal chain \[
\mathcal{U}_1\preceq \cdots \preceq \mathcal{U}_{\ell}
\] in $\mathcal P_n$, one can always build a merge tree that realizes the chain as follows. Start by defining a filtration of the set of subsets of $[n]$, where a subset $V\subset \mathbf{n}$ enters the filtration at $n+i$, where $i$ is the smallest index such that $V\subset U$ for some $U\in \mathcal{U}_i$. This defines a function from the set of subsets of $[n]$ (of which the geometric realization is the $n$-simplex) to $\R$. Taking the merge tree of this function as in Remark \ref{alternative_mt} associates a merge tree to a chain in $\mathcal P_n$.

Injectivity of the map from standard form merge trees to maximal chains is also clear. If two merge trees in standard form produce the same maximal chain, then their heights and adjacency relationships must be the same, i.e., they must be combinatorially equivalent.

The number of maximal chains in $\mathcal{P}_n$ was determined by Erd\H{o}s and Moon \cite{Moon} to be $(n+1)!n! 2^{-n}$.
This number is easily understood in the setting of merge trees. First, one chooses two of the $n+1$ connected components to merge at height $n+1$. Then one chooses two of the remaining $n$ connected components to merge at height $n+2$.
This process repeats until we run out of options at height $2n$. The number of ways of constructing standard form merge trees is thus
\[
\binom{n+1}{2}\binom{n}{2}\cdots \binom{2}{2}=\dfrac{(n+1)n}{2}\cdot \dfrac{n(n-1)}{2}\cdots \dfrac{2\cdot 1}{2}=\dfrac{(n + 1)!n!}{2^{n}}.
\]
\end{proof}

\begin{example}
Figure \ref{latticetree} shows the lattice of partitions on the set $\{0,1,2\}$ together with the three possible merge trees corresponding to the maximal chains in the lattice.

\begin{figure}
    \centering
    \includegraphics[width=\textwidth]{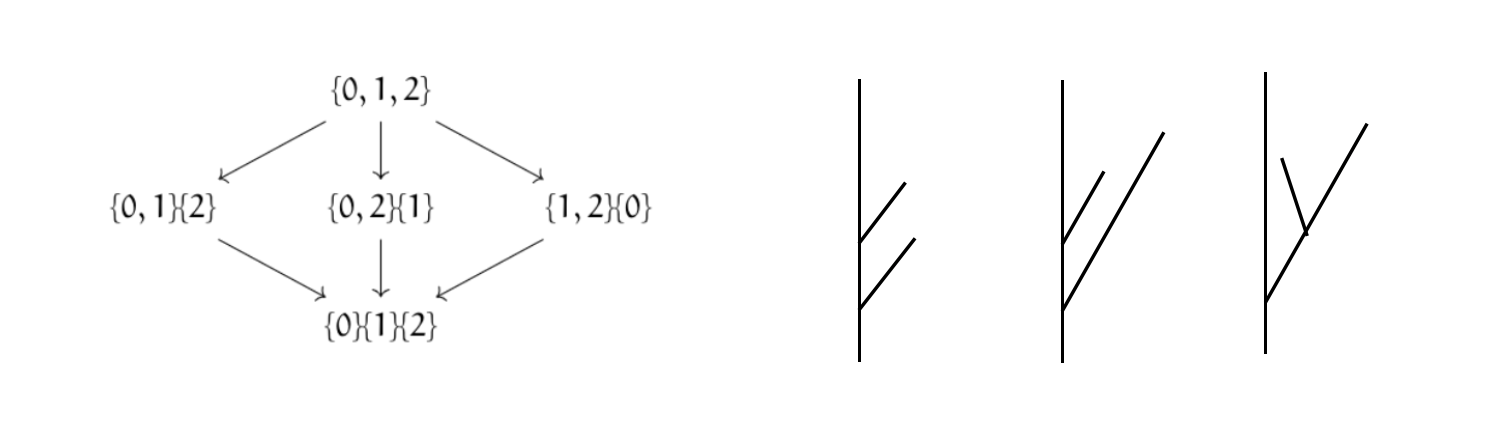}
    \caption{Left: The lattice of partitions of the set $\{0,1,2\}$. Right: The three possible merge trees corresponding to the maximal chains in the lattice, illustrating Theorem \ref{MaxChain}.}
    \label{latticetree}
\end{figure}
\end{example}

\begin{remark}[Expected Tree Realization Number]\label{rmk:expected-TRN}
It is very convenient that $n!$ appears in the numerator of the sum of realization numbers.
As we explain in greater depth in the section on statistics for the realization number, this allows us to compute the average realization number when $S_n$ is equipped with the uniform measure, for which the probability of a permutation $\sigma$ $P(\sigma)=\dfrac{1}{n!}$.
Indeed, by rearranging terms slightly, we see that the expected realization number is determined by the ratio of $(n+1)!$ and $2^n$:
\[
\mathbb{E}[R]=\sum_{\sigma\in S_n} R(\sigma) P(\sigma)=\dfrac{1}{n!}\dfrac{(n + 1)!n!}{2^{n}}=\dfrac{(n + 1)!}{2^{n}}.
\]
\end{remark}
Before studying the probabilistic aspects of the realization number more fully, we first compare  Theorem \ref{MaxChain} with analogous counting results for phylogenetic trees in the next section.

\subsection{Counting Merge Trees versus Phylogenetic Trees} \label{sec_MT_vs_PT}

In this section, we compare two counting results for combinatorial merge trees and for phylogenetic trees. 
On the one hand, Theorem \ref{MaxChain} implies that there are $\dfrac{(n + 1)!n!}{2^{n}}$ different combinatorial merge trees with $n+1$ leaves.  On the other hand, it was shown in  \cite{NumberTrees} that there are $(2n-1)!!$ distinct combinatorial phylogenetic trees with $n+1$ leaves.
In general, there are more classes of merge trees than there are phylogenetic trees.
In the next example, we work through the case $n=3$ in detail.

\begin{example}\label{ex:15-vs-18-trees}
For $n=3$, i.e.,~$4$ leaf nodes, these formulas imply that there are 18 different classes of merge trees, but only 15 classes of phylogenetic trees, shown in Figure \ref{fig:15-binary-trees}.
In Figure \ref{cayley_averagetree}C, one can see the 18 different classes of merge trees, arranged by row according to their permutation type in $S_3$.
There are three pairs of merge trees highlighted with colored boxes that correspond to the same combinatorial type of phylogenetic tree.

\begin{figure}[h]
    \centering
    \includegraphics[width=.6\textwidth]{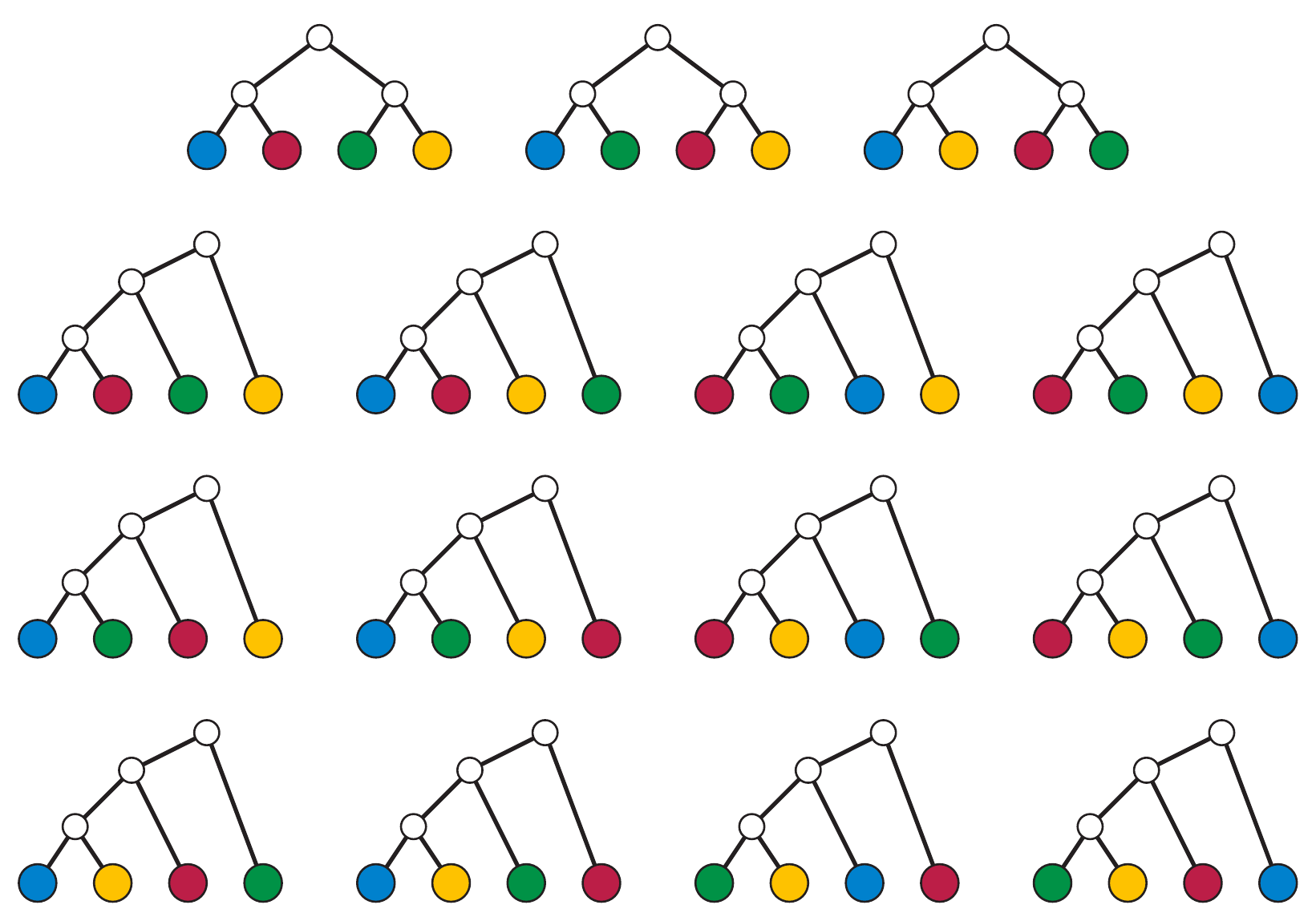}
    \caption{[Wikipedia, "Double Factorial", "Unordered binary trees with 4 leaves", n.d.]. The 15 different binary rooted trees with four labelled-by-color leaf nodes. The top node should be regarded as the unique child of the root node. This should be compared with the 18 different merge trees in Figure \ref{cayley_averagetree}C, as discussed in Example \ref{ex:15-vs-18-trees}.}
    \label{fig:15-binary-trees}
\end{figure}

\end{example}

As the example above shows, the essential difference between classes of merge trees and classes of phylogenetic trees is that merge trees are sensitive to relative heights of internal (death) nodes, whereas a phylogenetric tree is not.
This also explains why two combinatorially equivalent metric phylogenetic trees $(T,m)$ and $(T',m')$ may be associated to different permutation types, if one uses Proposition \ref{prop:BHV-MT-map-compare} to define a height function on each and compute a barcode according to the Elder rule.
However there are certain orders of births and deaths that must be preserved.
As one can see in Figure \ref{cayley_averagetree}C, the pair of trees in the purple box under column B both have the blue bar being born before and dying after the purple bar; the relative positioning of the death time associated to the red bar is the only thing that changes.

In this section we pinpoint more precisely how many different classes of merge trees can produce the same class of phylogenetic tree. 
As one might imagine, this is dictated in part by certain subgroups of the symmetric group, determine essentially by the number of incomparable internal nodes in a certain the natural partial order on the tree nodes specified by $p<q$ if $p$ is on the unique path from $q$ to the root.
Our bound on the number of classes of merge trees that define the same class of phylogenetic trees is formulated as follows. Recall that we assume that the root of any rooted tree has a unique child.

%\textbf{Nota Bene: I excised some of Brendan's labelling procedure as it didn't seem to be completely necessary for the proof of the following proposition and seemed to repeat some earlier results. We should revise if the following proof does not read well.}

\begin{proposition}\label{prop:combo-difference-bound}
Let $T$ be a combinatorial phylogenetic tree. Let $c$ denote the unique child of the root vertex.
Let $A_i$ be the set of internal nodes of $T$ that are $i$ hops away from $c$ in the path metric (in particular, $A_0=\{c\}$).

If $\eta(T)$ denotes the number of combinatorial equivalence classes of merge trees indistinguishable from $T$ when regarded as combinatorial phylogenetic trees, then $$\eta(T)\geq \prod_{j=0}^k |A_j|!.$$
\end{proposition}
\begin{proof}
We prove our result by induction on  the maximum path distance in $T$ from the child $c$. 
If the maximum path distance to the child is 0, then $T$ has a unique internal node $c$, i.e., $T$ has three nodes: the root $r$, its child $c$, and two leaves. This tree admits unique combinatorial merge and phylogenetic strucures, whence $\eta(T)=1=0!$. 

Suppose now the result holds whenever the maximal path distance from the child $c$ is less than $k$, for some $k\geq 1$.  Decompose the internal nodes of $T$ into $k$ sets $A_1,A_2,...,A_k$. All nodes in $A_k$ have only (two) leaf descendents, as otherwise there would exist an internal node further away from $c$ than some node in $A_k$, so the maximal path distance to $c$ would be greater than $k$. 

Let  $A_k=\{q_1,q_2,...,q_s\}$. If we remove the leaf nodes attached to each $q_i\in A_k$, we obtain a phylogenetic tree $T'$ with internal nodes partitioned into sets $A_1,A_2,...,A_{k-1}$.  By the induction hypothesis, there are at least $\prod^{k-1}_{j=1}|A_j|!$ combinatorial equivalence classes of merge trees indistinguishable from $T'$ when considered as phylogenetic trees. 

For each such equivalence class, we can obtain merge trees indistinguishable from $T$ as phylogenetic trees by reattaching the leaves to each $q_i$ and choosing any ordering on $A_k$, which we may do because all $q_i$ are at the same distance from $c$, and hence are incomparable nodes. Since there are $|A_k|!$ possible total orders on the set of $q_i$, we can conclude.
\end{proof}
 %\qed

%%%%%%%%%%%%%%%%%%%%%%%%%%%%%%%%%%%%%
%%%%%%%%%%%%%%%%%%%%%%%%%%%%%%%%%%%%%
%%%%%%%%%%%%%%%%%%%%%%%%%%%%%%%%%%%%%
%%%%%%%%%%%%%%%%%%%%%%%%%%%%%%%%%%%%%

\section{The Probabilistic Study of Tree Realization Numbers}\label{sec_statistics_real}

As already foreshadowed by Remark \ref{rmk:expected-TRN}, the formula in Theorem \ref{MaxChain} provides us with an unexpected gift in the study of statistics for realization numbers. Assuming that every combinatorial type of barcode is equally likely, so that each permutation type $\sigma$ has probability $\dfrac{1}{n!}$, we calculated that the expected tree realization number (TRN) is
\[
E[R]=\sum_{\sigma\in S_n} R(\sigma) P(\sigma)=\dfrac{1}{n!}\dfrac{(n + 1)!n!}{2^{n}}=\dfrac{(n + 1)!}{2^{n}}.
\]

We regard the assumption that each barcode permutation type is equally likely as a sort of ``null hypothesis'' to be tested against.
Even if one considers Gaussian perturbations to functional data, characterizing the image of this measure on the space of merge trees and hence (combinatorial types) of barcodes is an open problem.
Depending on the setup, it may be the case that features tend to die in the order in which they are born (a sort of ``topological first in first out'' queue) or it might be the case that features die in the opposite order in which they are born (a ``first in last out'' queue). In general, for real data, it is unlikely that the distribution of permutation types of (barcodes of) merge trees will be uniform. Regardless, characterizing the distribution of  TRNs in terms of the output of the function $R:S_n \to \N$ when $S_n$ is equipped with the uniform measure provides an important null hypothese against which to test real data. 

In this section we start with a brief outline of computational methods for generating random barcodes and compare the corresponding distribution of permutation types with the uniform distribution.
We then provide formulas for first and second moments of the pushforward distribution $\pi_n:=R_*\mu_n$, where $\mu_n$ is the uniform measure on $S_n$.
This allows us to calculate the variance of the TRN, which opens the door to hypothesis testing wherever the map from trees to barcodes is of interest to scientific applications.

Somewhat surprisingly, Theorem \ref{thm:Dirichlet-distribution} says that the exact value for the measure $\pi_n$ can be determined from $\pi_{n-1}$ and Dirichlet convolution with the uniform distribution on $S_{n-1}$, enabling us to study the entire distribution of TRNs as the number of features varies.
%A remark on how to compute all the higher moments is also offered.
To conclude, we provide a novel closed-form formula for the expected log-realization number, which allows us to characterize the empirical data in Figure \ref{realization_bio} in a more analytical manner.

\subsection{Distributions of Randomly Generated Barcodes}\label{sec_dist_barcodes}

\begin{figure}
    \centering
    \includegraphics[scale=0.7]{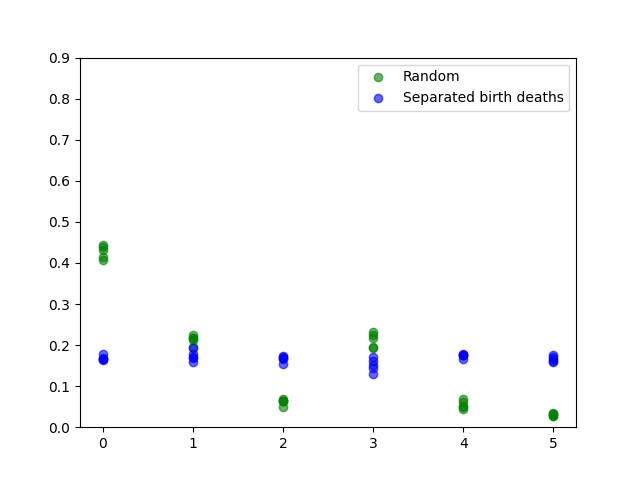}
    \caption{Two distributions on $S_3$ induced by distributions of barcodes with three non-essential bars. The elements of $S_3$ are indexed by the integers $0,...,5$. Green: we first pick uniformly the birth times $b_i$ in the interval $[0,100]$, then choose uniformly the death times $d_i \in [b_i,100].$ Blue: We pick uniformly three birth times $b_i \in [0,49]$ and three death times $d_i \in [50,100],$ which induces a uniform distribution on the symmetric group $S_3$.}
    \label{distribution_sn}
\end{figure}

In this section we briefly describe two methods to generate random barcodes and consider the pushforward distribution on $S_n$ for each of these. 
This pushforward is defined by the identification of barcodes with permutations as described above.

The first method was used in \cite{TRN} to generate barcodes and compare their realization numbers to the ones of biological barcodes, in a way similar to Figure \ref{realization_bio}. To generate a barcode with $n$ bars, for each bar we first pick a birth time $b_i$ uniformly at random in the interval $[0,100]$ and then pick a death time $d_i \in [b_i,100]$ uniformly at random. 
Because the latter distribution is conditioned on $d_i>b_i$, the induced distribution on the symmetric group is not uniform, as seen in Figure \ref{distribution_sn} with the ``random'' green dots. 

The second method displayed in Figure \ref{distribution_sn} forces separation of births and deaths to guarantee a uniform distribution on the symmetric group.
To generate $n$ bars in a barcode, we first choose $n$ births uniformly in the interval $[0,49]$, then $n$ death times $d_i$ uniformly in $[50,100]$. 
A moment of reflection shows that this provides a uniform distribution on $S_n$, as seen in Figure \ref{distribution_sn} with the ``separated'' blue points.

\subsection{The Distribution of Tree Realization Numbers via Dirichlet Convolution}

Let $\mu_n$ denote the uniform distribution on $S_n$. 
%By our correspondence, this is also a distribution on strict barcodes. 
By our correspondence, this is also a distribution on combinatorial equivalence classes of barcodes.
The tree realization number $R:S_n \to \N$ then defines a random variable where the probability $P(R=t)$ is determined by the number of permutations $n_t$ with realization number $t$.
The following theorem states that this probability can be computed recursively via convolution with the uniform distribution on ${1,\ldots,n}$.

\begin{theorem}\label{thm:Dirichlet-distribution} 
For any $k\geq 1$, let $\mu_k$ denote the uniform distribution on $S_n$ and  $\pi_n=R_*(\mu_n)$ its pushforward onto $\mathbb{N}$ via $R: S_n\to \mathbb N$. Let $U_k$ denote the uniform distribution on $\{1,2,...,k\}$.

The probability mass function of $\pi_n$ can be recursively defined as follows.
\begin{itemize}
    \item $\pi_1=U_1$.
    \item For $k>1$, $\pi_k=U_k*\pi_{k-1}$, where $*$ indicates Dirichlet convolution, i.e,. $$\pi_k(c)=\sum_{ab=c}U_k(a)\pi_{k-1}(b) \text { for all } c\in \mathbb N.$$
\end{itemize}
\end{theorem}

It follows immediately from this theorem that
$$\pi_n=U_n*U_{n-1}*...*U_1$$
for all $n\geq 1$.

\begin{proof}
We prove this theorem by induction on $k$. 
It holds trivially for $k=1$. Suppose that it holds for $k-1$ for some $k\geq 2$. 
Each number that has positive probability under $\pi_{k-1}$ corresponds to $R(\sigma)=\prod_{i=1}^{k-1}l_i(\sigma)$ for some $\sigma\in S_{k-1}.$ 

Consider the map $\kappa_{k-1}^j:S_{k-1}\rightarrow S_k$ that embeds $S_{k-1}$ into $S_k$ as follows. For every $\sigma\in S_{k-1}$, the permutation $\kappa_{k-1}^j(\sigma)$ is specified by
% \textbf{I THINK THE CONDTIONALS IN THIS DEFINITION ARE SLIGHTLY WRONG.}
% \[
%   \kappa_{k-1}^j(\sigma(i)) =
%   \begin{cases}
%                                     \sigma(i) & \text{if $i<j$} \\
%                                   \sigma(i)+1 & \text{if $k>i\geq j$} \\
%                                     j & \text{if $i=k$}
%   \end{cases}
% \]
% \textbf{I THINK THE FOLLOWING DEFINITION IS THE INTENDED ONE.}
\[
  \kappa_{k-1}^j(\sigma)(i) =
  \begin{cases}
                                    j & \text{if $i=k$} \\
                                   \sigma(i)+1 & \text{if $\sigma(i)\geq j$} \\
                                    \sigma(i) & \text{if $\sigma(i)<j$}.
  \end{cases}
\]
In other words, $\kappa^j_{k-1}$ sends $\sigma\in S_{k-1}$ to the permutation $\kappa^j_{k-1}(\sigma)\in S_k$ that maps the $k$-th object to $j$ and then ``bumps up'' by one the assigned value of elements in $\{1,2,...,k-1\}$ that are mapped to an element greater than or equal to $j$. 

Each map in the collection $\{\kappa_{k-1}^j\}_{j=1}^k$ is injective and collectively their images surject onto $S_k$. To determine the realization numbers for $S_k$, we therefore need only compute the realization number of $\kappa^j_{k-1}(\sigma)$ for all $j\in \{1,\ldots,k\}$ and $\sigma\in S_{k-1}$. 

Consider $R\big(\kappa^j_{k-1}(\sigma)\big)=\prod_{i=1}^k l_i\big(\kappa^j_{k-1}(\sigma)\big)$. Since $l_i(\sigma)=|\{r \leq i \mid \sigma(r)\geq \sigma(i)\}|$ for any permutation $\sigma$, it follows that 
$$l_i\big(\kappa^j_{k-1}(\sigma)\big)=l_i(\sigma)$$
for all $i<k$. On the other hand, since $r\leq k$ for all $r\in \{1,..,k\}$, 
$$|\{r \leq k\mid \kappa^j_{k-1}(\sigma(r))\geq \kappa^j_{k-1}(\sigma(k))\}|=|\{r\mid \sigma(r)\geq j\}|=k-j+1.$$ 
We conclude that $R\big(\kappa^j_{k-1}(\sigma)\big)=(k-j+1)R(\sigma)$. 

By the construction of $\kappa^j_{k-1}$,  $$\mu_k=\frac{1}{k}\sum^k_{j=1}(\kappa^j_{k-1})_*(\mu_{k-1}),$$ 
where $(\kappa^j_{k-1})_*(\mu_{k-1})$ is the pushforward of $\mu_{k-1}$ by $\kappa^j_{k-1}$, since each pushforward assigns mass $\frac{1}{(k-1)!}$ to each element of a unique subset of size $(k-1)!$ in $S_k$. 

We are now prepared to compute $\pi_k$. Let $x\in \mathbb{N}$.

\begin{align}
    \pi_k(x)&=R_*(\mu_k)(x)=\mu_k\big(R^{-1}(x)\big)\\ &=\frac{1}{k}\sum^k_{j=1}(\kappa^j_{k-1})_*(\mu_{k-1})\big(R^{-1}(x)\big)\\ &=\frac{1}{k}\sum^k_{j=1}\mu_{k-1}\Big((\kappa^j_{k-1})^{-1}\big(R^{-1}(x)\big)\Big)\\
&=\frac{1}{k}\sum^k_{j=1}\mu_{k-1}\Big((\kappa^j_{k-1})^{-1}\big(\{\sigma\in S_k\mid R(\sigma)=x\}\big)\Big)\\ 
&=\frac{1}{k}\sum^k_{j=1}\mu_{k-1}\big(\{\tilde{\sigma}\in S_{k-1}\mid (k-j+1)\cdot R(\tilde{\sigma})=x\}\big)\\ 
&=\frac{1}{k}\sum^k_{j=1}\mu_{k-1}\big(\{\tilde{\sigma}\in S_{k-1}\mid j\cdot R(\tilde{\sigma})=x\}\big)\\ 
&=\frac{1}{k}\sum_{jb=x}\mu_{k-1}\big(\{\tilde{\sigma}\in S_{k-1}\mid R(\tilde{\sigma})=b\}\big)\mathbbm{1}_{[k]}(j)\\ 
&=\sum_{ab=x}\mu_{k-1}\big(\{\tilde{\sigma}\in S_{k-1}\mid R(\tilde{\sigma})=b\}\big)\frac{\mathbbm{1}_{[k]}(a)}{k}\\
&=\sum_{ab=x}U_k(a)\pi_{k-1}(b),
\end{align}
where the second line follows from the identity $\mu_k=\frac{1}{k}\sum^k_{j=1}(\kappa^j_{k-1})_*(\mu_{k-1})$, the fifth line follows from $R(\kappa^j_{k-1}(\sigma))=(k-j+1)R(\sigma),$ and the sixth and seven lines are simple changes of variables.
\end{proof}

For what follows, it is useful  to consider for each $n$ the multiset $\Pi_n$, which is the range of $R:S_n\rightarrow \mathbb{N}$, taking into account multiplicities. 
Let $m_n:\mathbb N\rightarrow \mathbb{Z}_{\geq 0}$ be the multiplicity function of $\Pi_n$, i.e., $m_n(x)$ is the number of times $x\in \mathbb{N}$ appears in $\Pi_n$, which is the number of permutations in $S_n$ that have realization number $x$.  In particular, $m_n(x)=0$ if and only if $x\not\in \Pi_n$. 

Since $\pi_n$ is the pushforward of the uniform distribution on $S_n$, the probability of each $x$ is determined by dividing the multiplicity function by $n!$, i.e.,~$\pi_n(x)=\frac{m_n(x)}{n!}.$ 
The following corollary follows directly from the construction of $\pi_n$.

\begin{corollary}The multiset $\Pi_n$ can be constructed recursively as follows:

\begin{itemize}
    \item $\Pi_1=\{1\}$.
    \item For $i>1$, $\Pi_i$ is the multiset with multiplicity function specified by $$m_n(x)=\sum_{ab=x}m(b)\mathbbm{1}_{[i]}(a)\mathbbm{1}_{\Pi_{i-1}}(b).$$
\end{itemize}
\end{corollary}

In other words, $\Pi_n$ can be defined as a $[k]*\Pi_{n-1}$, where $[k]=\{1,\ldots,k\}$ and $*$ is the Dirichlet convolution of multisets. 

\begin{figure}[h]
    \centering
    \includegraphics[width = .8\textwidth]{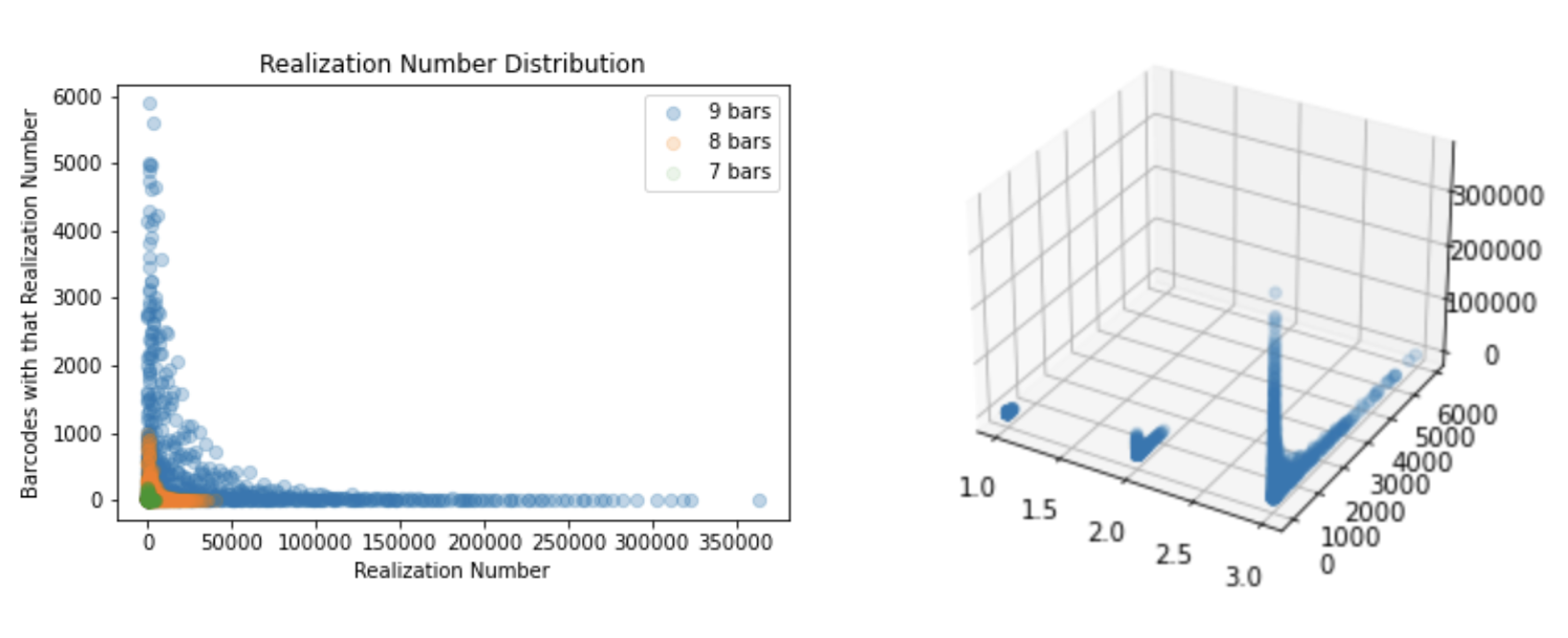}
    \caption{Distribution of Realization Numbers for 7,8,9 bars.}
    \label{fig:barsdist}
\end{figure}

\begin{example}
We now explicitly describe $\Pi_i$ for $i\in \{1,2,3,4\}$. For convenience we write the mutisets $\Pi_i$ as sets with repetition. Counting the number of appearances of a number $k$ determines $m_i(k)$.

\begin{align}\Pi_1&=\{1 \} \\
\Pi_2&=\{1,2\}\\
\Pi_3%=\{1,2,2,4,3,6\}\\
\Pi_4&=\{1,2,2,2,3,3,4,4,4,4,6,6,6,6,8,8,8,9,12,12,12,16,18,24\}
\end{align}
\end{example}

To conclude this section, we consider the moments of $\pi_n$. We explicitly calculate its first and second moments, obtaining the mean and variance of $\pi_n$ as corollaries, and outline a general formula for the higher moments.
\begin{proposition}\label{expected_real_number}
$\mathbb{E}(\pi_n)=\frac{(n+1)!}{2^n}$.
\end{proposition}
\begin{proof}
This is the content of Remark \ref{rmk:expected-TRN}, which establishes this proposition as a consequence of Theorem \ref{MaxChain}.
\end{proof}

\begin{proposition}
$\mathbb{E}(\pi_n^2)=\frac{(n+1)(2n+1)!}{12^n}.$
\end{proposition}

\begin{proof}
%Let $\mathcal{B}_n$ be the set of strict barcodes with $n$ non-essential bars. 
We prove the result by induction on $n$. The base case ($n=1$) holds trivially, so assume that the formula holds for $n=k$. 

Consider $\mathbb{E}(\pi_{k+1}^2)=\sum_{b\in B_n}R(b)^2$. Since
$$(k+1)!\sum_{b\in B_n}R(b)^2=(k+1)!\mathbb{E}(\pi_{k+1}^2)=\sum_{x\in \mathbb{N}}m_{k+1}(x)^2,$$ 
to prove our result, we need only show that 
$$\sum_{x\in \mathbb{N}}m_{k+1}(x)^2=\frac{(n+1)!(2n+1)!}{12^n}.$$ 
We call the quantity on the left $\mathbb{E}(\Pi_{k+1}^2)$: $$\mathbb{E}(\Pi_{k+1}^2)=\sum_{b\in B_{k+1}}R(b)^2=\sum_{a=1}^{k+1}\sum_{b\in B_k} (aR(b))^2=\sum_{a=1}^{k+1}a^2\sum_{b\in B_k}R(b)^2 =\sum_{a=1}^{k+1}a^2 \mathbb{E}(\Pi_{k}^2).$$ By the sum of squares formula, we can rewrite this as $$\bigg(\frac{(k+1)(k+2)(2k+3)}{6}\bigg)\mathbb{E}(\Pi_{k}^2)$$ $$=\bigg(\frac{(k+1)(k+2)(2k+3)(2k+2)}{(2k+2)6}\bigg)\mathbb{E}(\Pi_{k}^2)$$ $$=\bigg(\frac{(k+2)(2k+2)(2k+3)}{2*6}\bigg)\mathbb{E}(\Pi_{k}^2)$$ $$=\bigg(\frac{(k+2)(2k+2)(2k+3)}{12}\bigg)\bigg(\frac{(k+1)!(2k+1)!}{12^k}\bigg)$$ $$\frac{(k+2)!(2k+3)!}{12^{k+1}}=\frac{((k+1)+1)!(2(k+1)+1)!}{12^{k+1}}.$$ 
\end{proof}

\begin{corollary}
$\mathbb{V}(\pi_n)=\mathbb{E}(\pi_n^2)-\mathbb{E}(\pi_n)^2=\frac{1}{n!}\bigg(\frac{(n+1)!(2n+1)!}{12^n}-\frac{(n!(n+1)!)^2}{n!4^n}\bigg).$
\end{corollary}

\begin{remark}[Higher Moments of the TRN]
In general, we can define the $k$-th moment $\mathbb{E}(\pi_n^k)$ by rewriting $n!\mathbb{E}(\pi_n^k)=\mathbb{E}(\Pi_n^k)=(\sum^{n}_{a=1}a^k)\mathbb{E}(\Pi_n^{k-1})$ and using this recursive relationship to compute a formula. We note that by Faulhaber's formula, $$\sum^n_{a=1}a^k=\sum^k_{i=0}\frac{(-1)^{k-i}}{i+1} \binom{k}{i} B_{k-i}n^{i+1},$$ where $B_{k-i}$ is the $k-i$ Bernoulli number.
\end{remark}

One can view the results above as a complete characterization of TRNs under the null hypothesis that combinatorial classes of barcodes are distributed uniformly or as part of the growing literature on statistics on the symmetric group, see e.g., \cite{Kondorphdthesis}. 
In the following section, we investigate another such statistic. 

\subsection{Distributions of Log Realization Numbers}

Since the maximum realization number for a barcode with $n$ non-essential bars is $n!$, it is 
convenient to work instead with the logarithm of the realization number, which we call the log realization number.
The log realization number was used in \cite{TRN} as a statistic on barcodes obtained from dendrites; see Figure \ref{realization_bio} for a reminder. This was shown to distinguish between apical and cortical dendrites. 
Of course, the process of taking the logarithm affects the distribution of TRNs. Jensen's inequality provides a way to bound the expected log realization number.
In this section we compute the expected log realization number of uniformly drawn barcodes.

\begin{proposition}\label{prop:expected-log-TRN}
The expected log realization number for a combinatorial class $B$ of barcodes drawn from the uniform distribution on $S_n$ is
\[
\mathbb{E}_{\mu_n}\Big(\log\big(R(B)\big)\Big)=\sum^n_{i=1}\frac{\log(i!)}{i}.
\]
\end{proposition}

\begin{proof}
Recall that the set of left inversion vectors can be coordinatized as $[1]\times[2]\times...\times[n]$.
Since this Cartesian product has size $n!$, a uniform distribution on $S_n$ can be viewed as a uniform distribution on the set of left inversion vectors.
The notation $\mathbb{P}(B\sim \mu_n)$ denotes the probability of a combinatorial equivalence class of barcodes $B$ under the uniform distribution, that is  $\frac 1 {n!}$.  It follows that
\begin{eqnarray*}
\mathbb{E}_{\mu_n}\Big(\log\big(R(B)\big)\Big) &=& \sum_{B\in S_n}\log(\prod^n_{i=1}l_i(B))\mathbb{P}(B\sim \mu_n) \\
&=&\frac{1}{n!}\sum_{B\in [1]\times...\times[n]}\sum^n_{i=1}\log(l_i(B)) \\
&=& \frac{1}{n!}\sum^n_{i=1}\sum_{B\in [1]\times...\times[n]}\log(l_i(B)).
\end{eqnarray*}
% $$\mathbb{E}_U(\log(R(b))=\sum_{b\in S_n}\log(\prod^n_{i=1}l_i(b))\mathbb{P}(b\sim U)=$$ $$\frac{1}{n!}\sum_{b\in [1]\times...\times[n]}\sum^n_{i=1}\log(l_i(b))=\frac{1}{n!}\sum^n_{i=1}\sum_{b\in [1]\times...\times[n]}\log(l_i(b)).$$ 
Since $B\sim \mu_n$, and each coordinate in $[1]\times[2]\times...\times[n]$ is independent, the interior sum (for fixed $i$) is equal to $\frac{n!}{i}(\log(1)+\log(2)+...+\log(i))$. 
Hence 
\[
\mathbb{E}_{\mu_n}(\log(R(b)))=\sum^n_{i=1}\frac{\log(1)+\log(2)+...+\log(i)}{i}=\sum^n_{i=1}\frac{\log(i!)}{i}.
\]
\end{proof}

%%%%%%%%%%%%%%%%%%%%%%%%%%%%%%%%%%%%%%%%%%%%%%%%%%%%%%%%%%%%%%%%%%%%%%%%%

\section{Conclusion}

In this paper, we extended the results of \cite{TRN} and \cite{curry2017fiber} to provide a more precise characterization of the distribution of tree realization numbers (TRNs).
This investigation led us to consider the uniform distribution on the symmetric group and the expected TRN, which in turn put us in a setting where classical results from combinatorics could be used.
This extraction of the notion of a combinatorial version of a merge tree led us to understand more precisely the difference between merge trees and metric phylogenetic trees \cite{BHV}.

We emphasize that the TRN provides a convenient summary statistic on the space of barcodes that could lead to a better understanding of inherent biological properties of neurons. 
If we can identify where biological barcodes live on the space of barcodes, it opens the door to many applications such as statistics of learning of ``biological'' barcodes, allowing to create artificial barcodes that mimic the properties of biological ones and hence to generate neurons from them that are statistically relevant, yet express higher variability. 
By studying the simplest possible version of a null hypothesis---where combinatorial equivalence classes of barcodes are uniformly distributed---we are in a position to move on to study more interesting variants on the null hypothesis in TDA and explore the geometry of barcode space in even greater detail.
% In \cite{Kanari2020.04.15.040410}, they describe an algorithm to generate artificial neurons from biological barcodes. 
% The motivation here is to go one step further down the persistence pipeline by characterizing  understanding the biological barcodes to generate artificial neurons from artificial barcodes.  
\section{Aknowledgments}

JC would like to acknowledge NSF Grant
CCF-1850052 and NASA Contract 
\\
80GRC020C0016 for supporting his research. LK was supported by funding to the Blue Brain Project, a research center of the École polytechnique fédérale de Lausanne (EPFL), from the Swiss government’s ETH Board of the Swiss Federal Institutes of Technology. AG and KH gratefully acknowledge the support of Swiss National Science Foundation, Grant No.  CRSII5\_177237.

\bibliography{mybib}{}
\bibliographystyle{plain}

\end{document}